\documentclass{amsart}

\usepackage[english]{babel}
\usepackage{array,booktabs,float}
\usepackage{amsmath,amssymb,amsfonts}
\usepackage{amsthm}
\usepackage{bm}
\usepackage{amscd}
\usepackage{mathtools}
\usepackage{mathrsfs}
\usepackage{pb-diagram}
\usepackage{booktabs}

\theoremstyle{definition}
\newtheorem{definition}{Definition}[section]
\newtheorem*{definition*}{Definition}
\newtheorem{theorem}[definition]{Theorem}
\newtheorem*{theorem*}{Theorem}
\newtheorem{proposition}[definition]{Proposition}
\newtheorem*{proposition*}{Proposition}
\newtheorem{lemma}[definition]{Lemma}
\newtheorem*{lemma*}{Lemma}
\newtheorem{remark}[definition]{Remark}
\newtheorem*{remark*}{Remark}

\newtheorem*{example*}{Example}

\newtheorem*{problem*}{Problem}
\newtheorem{corollary}[definition]{Corollary}
\newtheorem*{corollary*}{Corollary}

\newcommand{\Z}{\mathbb{Z}}
\newcommand{\R}{\mathbb{R}}
\newcommand{\C}{\mathbb{C}}
\newcommand{\q}{\mathbb{H}}
\newcommand{\ind}{\text{ind}}
\newcommand{\Slash}[1]{{\ooalign{\hfil/\hfil\crcr$#1$}}}

\begin{document}

\title[Localization of a $KO^{\ast}(\text{pt})$-valued index ]{Localization of a $KO^{\ast}(\text{pt})$-valued index and the orientability of the $Pin^-(2)$ monopole moduli space}
\author{Jin Miyazawa}
\address{Graduate School of Mathematical Sciences, the University of Tokyo, 3-8-1 Komaba, Meguro, Tokyo 153-8914, Japan}
\email{miyazawa@ms.u-tokyo.ac.jp}
\maketitle
\begin{abstract}
It is known that the Dirac index of a $Spin^c$ structure is localized to the characteristic submanifold. 
We introduce the notion of $G^{\pm}(n,s^+,s^-)$ structure on a manifold as a common generalization of the $Spin^c$ structure and the $H_n(s)$ structure defined by D.~Freed--M.~Hopkins,  
and formulate a version of characteristic submanifold for the $G^{\pm}(n,s^+,s^-)$ structure. We show that the $KO^*(pt)$-valued index associated with the $G^{\pm}(n,s^+,s^-)$ structure is localized to the characteristic submanifold. 
As an application, we give a topological sufficient condition for the moduli space of $Pin^-(2)$ monopoles to be orientable.

\end{abstract}
\tableofcontents
\section{Introduction}
In this paper, we introduce the notion of $G^{\pm}(n,s^+,s^-)$ structure, which is a generalization of the $Spin$ structure, the $Spin^c$ structure, the $Pin^{\pm}$ structure, and the $H_s(n)$ structure due to Freed and Hopkins \cite{FreedHopkins1}. We construct an elliptic differential operator associated with the $G^{\pm}(n,s^+,s^-)$ structure $\mathfrak s$, and we have its index $\ind(\mathfrak s)$ with values in $KO^{s^- -n-s^+}(pt)$. 
The index $\ind(\mathfrak s)$ is a generalization of the Atiyah-Milnor-Singer invariant which is defined by the Dirac type operator with the Clifford action for a $Spin$ manifold. 
The index $\ind(\mathfrak s)$ is also a generalization of the index of the $H_s(n)$ structure with values in $KO^{-n-s}(pt)$ defined by Freed and Hopkins. 

Our main theorem is that the index above is localized to a certain submanifold which is a generalization of characteristic submanifold of the $Spin^c$ structure (Theorem \ref{main}). 
Many people give generalizations of the localization of the index of the $Spin^c$ structure to the characteristic submanifolds. For example, W.~Zhang~\cite{Zh93}, J.L.~Fast--S.~Ochanine~\cite{FO04}, M.~Furuta--Y.~Kametani~\cite{FK00}, S.~Hayashi~\cite{Hayashi}. 
The methods of the localization of the index in the works of \cite{FO04} and in \cite{Hayashi} are to localize the topological index by using the excision theorem of $K$ or $KO$ theory respectively.  
Our method of the localization is based on a version of the Witten deformation, which is introduced by E. Witten \cite{Witten2} in 1982. The Witten deformation is an analytical counterpart of the excision. 

The main application of the main theorem is to give sufficient condition for the $Pin^-(2)$ monopole moduli space to be orientable, which enables us to refine the $Pin^-(2)$ monopole invariant. The $Pin^-(2)$ monopole invariant is a variant of the Seiberg-Witten invariant introduced by N. Nakamura (\cite{Nakamura1}, \cite{Nakamura2}). The orientability of the moduli spaces in gauge theory was originally studied for instanton by Donaldson~(\cite{Donaldson-ori},~\cite{Donaldson2}). 
Donaldson's argument can be applied in the case of the singular instantons introduced by P.B.~Kronheimer--T.S.~Mrowka~\cite{kronheimer2011knot} and in the case of the $U(n)$ instantons introduced by Kronheimer~\cite{kronheimer2005four}. 
In these cases, the moduli spaces are orientable. On the other hand, we show that the $Pin^-(2)$ monopole moduli space may be non-orientable. Strictly speaking, we have an explicit example of a $4$-manifold for which the determinant bundle on the ambient space of the moduli space is non-trivial (Corollary \ref{exexample}). We expect that our new method using the Witten deformation could be applied to other moduli spaces in gauge theory. Recently,  Joyce, Tanaka and Upmeire \cite{Joyce-Tanaka-Upmire} give a new framework to deal with the orientation of moduli spaces. It is an interesting problem to understand the relation between their argument and ours. 

This paper is organized as follows. In Section $2$, we establish our conventions and define the index of the $G^{\pm}(n,s^+,s^-)$ structure. In Section $3$, we first formulate the main theorem and give a proof in the rest of the section. The proof of the analytical detail of  the key localization is postponed to Appendix. In Section $4$, we describe two examples in details: Freed--Hopkins' $H_s(n)$ structure, and the $G^{+}(5,0,4)$ structure which we use in the next section. In Section $5$, as an application, we give a topological sufficient condition for the $Pin^-(2)$ monopole moduli space to be orientable and we give an example of a $4$-manifold for which the determinant bundle on the ambient space which contains the $Pin^-(2)$ monopole moduli space is non-trivial. 

\subsection{Acknowledgements}
This paper is based on the master thesis of the author. The author would like to show his deep appreciation to his adviser Mikio Furuta for helpful suggestions, useful discussions, and continuous encouragement during this work. The theme in this thesis is suggested by him. The author also thanks to Nobuhiro Nakamura for sending the note \cite{Nakamuranote} to him and for for useful discussions. The author thanks to Kiyonori Gomi for noticing the paper of Freed--Hopkins and had a discussion with him. The author thanks to Hokuto Konno and Masaki Taniguchi for helpful suggestions and continuous encouragement. The author was supported by JSPS KAKENHI Grant Number 21J22979 and WINGS-FMSP program at the Graduate school of Mathematical Science, the University of Tokyo.

\section{Definition of the index}
\subsection{The $G^{\pm}(n,s^+,s^-)$ structures} 
Our purpose here is to establish our notations and conventions. We follow the notations of \cite{LM89} for Clifford algebras and follow the definition of \cite{AtiyahSinger1} for $KO$ groups. 

\begin{itemize}\label{conv}

\item Let $Q_r$ denote the standard metric on $\R^r$. Let $Q$ denote the quadratic form on $\R^l \oplus \R^m$ defined by\[Q=Q_l-Q_m. \] We will denote by $Cl_{(l,m)}$ the Clifford algebra generated by $\{ v \in \R^l \oplus \R^m \}$ subject to $v^2=Q(v)$. We abbreviate $Cl_{(n,0)}$ and $Cl_{(0,n)}$ to $Cl_n$ and $Cl_{-n}$ respectively. We write $\epsilon_1, \dots \epsilon_l, e_1, \dots e_m$ for the standard generators of $Cl_{l,m}$. Here $\hat{\otimes}$ denotes the $\Z/2\Z$ graded tensor product. Then there is an isomorphism $Cl_{(l_1, m_1)} \hat{\otimes} Cl_{(l_2, m_2)} \cong Cl_{(l_1+l_2, m_1+m_2)}$.

\item Let $(Cl_{(l,m)})^0$ denote the even part of $Cl_{(l,m)}$ and $(Cl_{(l,m)})^1$ the odd part of $Cl_{(l,m)}$.

\item Let $X$ be a compact hausdorff space. A representative element of $KO^{(b,a)}(X)$ is given by a four-tuple :
\begin{enumerate}
\item The Clifford algebra $Cl_{(b,a)}$.
\item  A separable Hilbert space with a left $Cl_{(b,a)}$ action. 
\item  A continuous map $s \colon X \to \hom(H,H)$ (the topology of $\hom(H,H)$ is the operator norm) such that $s(x)$ is a bounded skew adjoint and Fredholm operator which anti-commutes with the $Cl_{(b,a)}$ action for all $x \in X$.  
\item A map $\epsilon \colon H \to H$ such that $\epsilon^2=1$ and anti-commutes with $s$ and the generators of $Cl_{(b,a)}$. 
\end{enumerate}
We will write $(s,\epsilon,Cl_{(b,a)},H)$ for this four-tuple. 

\item We define $(s,\epsilon,Cl_{(b,a)},H)$ and $(s',\epsilon',Cl_{(b,a)},H')$ are equivalent if they satisfy the following properties:

\begin{itemize}
\item There exist four tuples $(s_0,\epsilon_0, Cl_{(b,a)},H_0)$ and $(s_1,\epsilon_1,Cl_{(b,a)},H_1)$ such that $s_0, s_1$ are isomorphism. 
\item There exists an isometric linear map $f\colon H \oplus H_0 \to H' \oplus H_1$ such that $f \circ (\epsilon \oplus \epsilon_0)f^{-1}=(\epsilon' \oplus \epsilon_1)$ and $f \circ (s \oplus s_0) \circ f^{-1}$ and $s' \oplus s_1$ are homotopic through continuous maps from $X$ to $\hom(H' \oplus H_1, H' \oplus H_1)$ which anti-commute with $\epsilon' \oplus \epsilon_1$ and the $Cl_{(b,a)}$ action. 
\item The homotopy above anti-commutes with $\epsilon' \oplus \epsilon_1$ and Clifford action. 
\end{itemize}

\item If $s$ is an unbounded skew-adjoiont Fredholm operator and 
\[\tilde s= \frac{s}{\sqrt{1+ss^*}}\] 
satisfies the above properties, we write $(s,\epsilon,Cl_{(b,a)},H)$ instead of $(\tilde s,\epsilon,Cl_{(b,a)},H)$. 

\item We will denote by $(\rho_1, V_1)$ the $\Z/2\Z$ graded representation of $Cl_{(1,1)}$ which is given as follows. Let $V_1^0 =V_1^1 =\R, V_1 = V_1^0 \oplus V_1^1$ and \[\rho(\epsilon)=\begin{pmatrix} & 1 \\ 1 & \end{pmatrix}, \; \rho(e)=\begin{pmatrix} & -1 \\ 1 & \end{pmatrix}, \]  where $\epsilon, e (\epsilon^2=1, e^2=-1)$ are the generators of $Cl_{(1,1)}$. 

\item We will denote by $V_n :=V_1^{n\hat{\otimes}}$ and define $(\rho_n, V_n)$ to be the representation introduced by the natural isomorphism $Cl_{(n,n)} \cong Cl_{(1,1)}^{n\hat{\otimes}}$. Let $V^*_n$ denote the dual space of $V_n$ which has a natral right $Cl_{(n,n)}$ module. 
We will denote by $\epsilon^n$ the grading operator. 
\end{itemize}

\begin{lemma}\label{H=H'V}
Assume $a \ge b$. Let $(s,\epsilon,Cl_{(b,a)},H)$ be a representative element of $KO^{(b,a)}(X)$. Then there exists an element $(s',\epsilon',Cl_{(0,a-b)},H')$ of $KO^{(0.a-b)}(X)$ which satisfies the following properties : There is a isomorphism between the Hilbert spaces
\[f \colon H \to H' \hat{\otimes} V^*_{b}\]
such that $s=f \circ s' \otimes \epsilon^b \circ f^{-1}$, $\epsilon=f \circ \epsilon' \otimes \epsilon^b \circ f^{-1}$, $\epsilon_i=f \circ 1 \otimes \epsilon_i \circ f^{-1}$ (for $i=1, \dots, b$), $e_i=f \circ 1 \otimes e_i \circ f^{-1}$ (for $i=1, \dots, b$), and $e_{b+i} = f \circ e'_i \otimes \epsilon^b \circ f^{-1}$ (for $i=1, \dots, a-b$) where 
$e'_{1}, \dots, e'_{a-b}$ are the generators of $Cl_{(0,a-b)}$ and $\epsilon_1, \dots, \epsilon_b, e_1, \dots,e_b, e_{b+1}, \dots, e_{a}$ are the generators of $Cl_{(b,a)}$. 
\end{lemma}
\begin{proof}
The subspace $H' \subset H$ is given by the intersection of $+1$ - eigenspaces of $\epsilon_1 e_1, \dots, \epsilon_b e_b$. The actions of $s, \epsilon$ and $Cl_{(a-b)}$ preserve $H'$. We write $s',\epsilon',e'_{i}$ for restriction of the actions of them to $H'$. Then the existence of $f$ is straightforward. 
\end{proof}

\begin{definition}
We define a group $G^{\pm}(n,s^+,s^-)$ as 
\[G^{\pm}(n,s^+,s^-)=\{ g \in Cl_{\pm n} \hat{\otimes} Cl_{(s^+, s^-)} \mid g=v_{i_1}\dots v_{i_{2k}}, v_{i_j} \in \R^n \text{or} \;\R^{s^+} \text{or} \; \R^{s^-}, \lvert v_{i_j} \rvert =1  \}. \]
\end{definition}

\begin{definition}
Let 
\[
S(O(n)\times O(s^+)\times O(s^-))=\{ (A,B^+,B^-) \in O(n)\times O(s^+)\times O(s^-) \mid \det A \det B^+ \det B^- =1 \}. 
\]
The two-to-one homomorphism
\[
p \colon G^{\pm}(n,s^+, s^-) \to S(O(n)\times O(s^+)\times O(s^-))
\]
is defined by
\[
p(g)v := gvg^{-1}
\]
for
$g \in G^{\pm}(n,s^+, s^-),  v \in \R^n \oplus \R^{s^+} \oplus \R^{s^-}$. 

We will denote by $p_n, p_{s^+}$ and $p_{s^-}$ the compositions of $p$ with the projections from $S(O(n)\times O(s^+)\times O(s^-))$ to each component $O(n), O(s^+)$ and $O(s^-)$ respectively. 
\end{definition}

\begin{definition}\label{def-G-+}
Let $Y$ be an $n$ dimensional Riemannian manifold. Let us denote by $P_{Y}$ the orthogonal frame bundle of $TY$. 
We call $(\tilde{P},P,\pi,o,E_+, E_-)$ a $G^{\pm}(n,s^+,s^-)$ structure if it satisfies the following properties:
\begin{itemize}
\item $E_{\pm}$ is an $s^{\pm}$ dimensional real vector bundle such that their structure group is $O(s^{\pm})$. We will denote by $P_{E_{\pm}}$ its frame bundle. (Double-sign corresponds. ) 
\item $o$ is an orientation of $TY \oplus E_+ \oplus E_-$. 
\item $P$ is a principal $S(O(n)\times O(s^+)\times O(s^-))$ bundle defined as the subbundle  of $P_{Y} \times_{Y} P_{E_+} \times_Y P_{E_-}$:
\[
P=\{ (f_n,f_+,f_-) \in P_{Y} \times_{Y} P_{E_+} \times_Y P_{E_-} \mid f_n, \; f_+, \; f_- \text{are compatible with the orientation}\;o\;\text{in this order} \}.
\]
 
\item We denote by $\tilde{P}$ a principal $G^{\pm}(n,s^+,s^-)$ bundle and $\pi \colon \tilde{P} \to P$ is a smooth map such which satisfies the following commutative diagram 
\[
\begin{CD}
\tilde{P} @>{\cdot g}>> \tilde{P}  \\
@V{\pi}VV   @V{\pi}VV  \\
P @>>{\cdot p(g)}> P
\end{CD}
\]
for all $g \in G^{\pm}(n,s^+,s^-)$. 
\end{itemize}
\end{definition}

\begin{definition}
Let $(\tilde{P},P,\pi,o,E_+, E_-)$ and $(\tilde{P'},P',\pi',o',E'_+, E'_-)$ be $G^{\pm}(n,s^+,s^-)$ structures on $Y$. We define they are isomorphic if there exist  isomorphisms of principal bundles $f \colon P \to P'$ and $\tilde{f} \colon \tilde{P} \to \tilde{P'}$ such that the following diagram commutes. 
\[
\begin{CD}
\tilde{P} @>{\tilde{f}}>> \tilde{P'} \\
@V{\pi}VV @V{\pi'}VV \\
P @>{f}>> P' \\
@V{p_{n}}VV  @V{p_{n}}VV \\
P_{Y} @>{id}>> P_{Y}
\end{CD}
\]
\end{definition}

\subsection{Definition of spinor bundles}

\begin{definition}
A generalized $\Z/2\Z$ graded spinor representation of $G^+(n,s^+,s^-)$ is a pair $(\alpha, S)$ with the following properties: 
\begin{itemize}
\item $S$ is a real vector space with a metric and a $\Z/2\Z$ grading $S=S_0 \oplus S_1$. 
\item $\alpha$ is a representation $\alpha \colon G^{\pm}(n,s^+,s^-) \to O(S)$ such that for all $g \in G^{\pm}(n,s^+,s^-)$, $\alpha(g)$ preserves the $\Z/2\Z$ grading of $S$. 
\item The representation space $S$ has an $G^{\pm}(n,s^+,s^-)$ equivariant products
\begin{align*}
c'_n \colon \R^n \times S \to S, \\
c'_{s^+} \colon \R^{s^+} \times S \to S, \\
c'_{s^-} \colon \R^{s^-} \times S \to S 
\end{align*}
such that they are anti-commutative each other and odd. Here we define the action of $G^{\pm}(n,s^+,s^-)$ on $\R^n, \R^{s^+}, \R^{s^-}$ through the  projections $p_n, p_+, p_-$. 
\item The multiplications above gives a $Cl_{\pm n} \hat{\otimes} Cl_{(s^+, s^-)}$ module structure on $S$. 
\end{itemize}
Moreover, if $S$ has an additional left Clifford action of $Cl_{(b,a)}$ and its generators anti-commutes with $c'_n,c'_{s^+},c'_{s^-}$, we call $(\alpha, S)$ a generalized $\Z/2\Z$ graded spinor representation with $Cl_{(b,a)}$ action. 
\end{definition}

\begin{definition}
We define $(\rho, Cl_{\pm n} \hat{\otimes} Cl_{(s^+, s^-)})$ to be the generalized $\Z/2\Z$ graded spinor representation with left $Cl_{\mp n} \hat{\otimes}Cl_{(s^-,s^+)}$ action as follows:
Let $g \in G^{\pm}(n,s^+,s^-)$ and $\phi \in Cl_{\pm n} \hat{\otimes} Cl_{(s^+, s^-)}$. We define  $\rho(g)\phi:=g\phi$  where the right hand side is a multiplication of $Cl_{\pm n} \hat{\otimes} Cl_{(s^+, s^-)}$. 
The representation $\rho$ preserves the $\Z/2\Z$ grading of $Cl_{\pm n} \hat{\otimes} Cl_{(s^+, s^-)}$. 
We define 
\[
c'_n \colon \R^n \times Cl_{\pm n} \hat{\otimes} Cl_{(s^+, s^-)} \to Cl_{\pm n} \hat{\otimes} Cl_{(s^+, s^-)}
\]
by $c'_n(v)\phi=v\phi$. Replacing $\R^n$ with $\R^{s_{\pm}}$, we define $c'_{s_+}, c'_{s_-}$ similarly. We define the additional left $Cl_{\mp n} \hat{\otimes} Cl_{(s^-, s^+)}$ action by  the following way. Let $\epsilon$ be the $\Z/2\Z$ grading operator of $Cl_{\pm n} \hat{\otimes} Cl_{(s^+, s^-)}$. We define $v \cdot \phi:=(\epsilon \phi)v$ for $v \in \R^n \oplus \R^{s^+} \oplus \R^{s^-}$ where the right hand side is a multiplication of $Cl_{\pm n} \hat{\otimes} Cl_{(s^+, s^-)}$. This defines left $Cl_{\mp n} \hat{\otimes} Cl_{(s^-, s^+)}$ because the right $Cl_{\pm n} \hat{\otimes} Cl_{(s^+, s^-)}$ action is odd and $v$ anti-commutes with $\epsilon$. 
\end{definition}

\begin{definition}
Let $\mathfrak{s}=(\tilde{P},P,\pi,o,E_+, E_-)$ be a $G^{\pm}(n,s^+,s^-)$ structure on $Y$ and  let $(\alpha,S)$ be a generalized $\Z/2\Z$ graded spinor representation of $G^+(n,s^+,s^-)$. 
Then they define the vector bundle 
\[
\Slash{S}=\tilde P \times_{\alpha} S. 
\]
Let $\Slash{S}_0$ and $\Slash{S}_1$ are subbundle of $\Slash{S}$ defined by $S_0$ and $S_1$ respectively. 
We define the Clifford multiplication on $\Slash S$ using $c'_n,c'_{s^+},c'_{s^-}$: 
\[
c_{TY} \colon TY \times \Slash S \to \Slash S, \; c_{E_{\pm}} \colon E_{\pm} \times \Slash S \to \Slash S.
\]
We call $\Slash{S}$ a generalized $\Z/2$ graded Spinor bundle. 
If $(\alpha, S)$ is a generalized $\Z/2\Z$ graded spinor representation with $Cl_{(b,a)}$ action, We call $\Slash{S}$ a generalized $\Z/2$ graded Spinor bundle with $Cl_{(b,a)}$ action. 
\end{definition}

\begin{remark}
If $\Slash{S}$ have a left $Cl_{(b,a)}$ action, we have a right $Cl_{(a,b)}$ action which commutes with $c_{TY}, c_{E_{\pm}}$. Let $\epsilon$ be a $\Z/2\Z$ grading operator and $\phi \in \Slash S$. For a generator $e \in Cl_{(b,a)}$, we define $\phi \cdot e \colon = e \cdot (\epsilon \phi)$. 
\end{remark}

\begin{definition}
Let $\mathfrak{s}=(\tilde{P},P,\pi,o,E_+, E_-)$ be a $G^{\pm}(n,s^+,s^-)$ structure on $Y$. We will denote by $\Slash{\mathfrak{S}} = \Slash{\mathfrak{S}}_0 \oplus \Slash{\mathfrak{S}}_1$ the generalized $\Z/2$ graded Spinor bundle with $Cl_{\mp n}\hat{\otimes}Cl_{(s^-, s^+)}$ action defined by the generalized $\Z/2\Z$ graded spinor representation with $Cl_{\mp n}\hat{\otimes}Cl_{(s^-, s^+)}$ action $(\rho, Cl_{\pm n} \hat{\otimes} Cl_{(s^+, s^-)})$. 
We call $\Slash{\mathfrak S}$ the standard Spinor bundle of $\mathfrak s$. 
\end{definition}

\begin{definition}\label{defineindex}
Let $\mathfrak{s}=(\tilde{P},P,\pi,o,E_+, E_-)$ be a $G^{\pm}(n,s^+,s^-)$ structure on $Y$ and $\Slash{\mathfrak S}$ the standard Spinor bundle of $\mathfrak s$. We define a Dirac type operator $\Slash D$ on $\Gamma(\Slash{\mathfrak S})$ by using the Clifford multiplication $c_{TY}$. We will denote by $\epsilon$ the $\Z/2\Z$ grading operator on $\Slash{\mathfrak S}$. 

Let us denote by $\ind (\mathfrak s)$ the element in $KO^{-n \pm (s^- - s^+)}(pt)$ defined by the representative element
\[(\Slash{D},\epsilon,  Cl_{\mp n} \otimes Cl_{(s^-, s^+)}, L^2(Y,\Slash{\mathfrak S})). \]
We call $\ind (\mathfrak s)$ the index of the $G^{\pm}(n,s^+,s^-)$ structure $\mathfrak s$. 
\end{definition}

\begin{remark}
The index defined above coincides the index of $Spin$ structures when $s^+=s^-=0$ so called Atiyah-Milnor-Singer invariant. This index is a generalization of the $\bmod 2$ index of $Spin$ structures on riemannian surface and the $\hat A$ genus. This is defined explicitely by Lawson-Michaelson in \cite{LM89}, II, \S 7. In some cases, the index coincides with the index of $H_s(n)$ structure introduced in \cite{FreedHopkins1}. 
\end{remark}

\section{Proof of the Main Theorem}
\subsection{Statement of the Main Theorem}
Now we can state the main theorem. 
\begin{theorem}\label{main}[The main theorem]
There exists a homomorphism $f$ such that the following diagram commutes :
\[
\begin{diagram}
\node{\Omega_{n}^{G^+(n,s^+,s^-)}(pt)}\arrow{e,t}{f} \arrow{se}\node{\Omega_{n-s^-}^{G^+(n-s^-, s^+, 0)}(pt)}\arrow{s} \\
\node[2]{KO^{-n-s^+ +s^-}(pt)} 
\end{diagram}
\]
where the morphisms to the $KO$ group are defined by the indecies. 
\end{theorem}

\begin{remark}
We can find a similar commutative diagram of $G^-$ bordism groups. In this case, the codomain target of $f$ is $\Omega_{n-s^+}^{G^-(n-s^+, 0, s^-)}(pt)$ instead of $\Omega_{n-s^-}^{G^+(n-s^-, s^+, 0)}(pt)$. The proof is parallel to that of the $G^+$ cases. Moreover, we can prove $G^+(n, s^+, s^-) \cong G^-(n, s^-, s^+)$ and hence we will study only the $G^+$ cases. 
\end{remark}

First, we will construct the morphism $f$ in the main theorem in this subssection. Second, we will prove a localization theorem of the indices in subsection \ref{Witten_deformation}. Some analytic lemmas are proved in Appendix. Finally, we prove the commutativity of the diagram using the localization theorem in the final part of this section and we will complete the proof of the main theorem.

\begin{lemma}\label{emb}
Let $\mathfrak{s}=(\tilde{P},P,\pi,o,E_+, E_-)$ be a $G^{\pm}(n,s^+,s^-)$ structure on $Y$. 
Then there exists a canonical $Spin$ structure of the vector bundle 
\[
(\det TY \otimes \det E_+) \oplus TY \oplus E_+ \oplus \det E_- \oplus E_-
\]
introduced by $\mathfrak s$. Conversely, if  
the vector bundle above has an $Spin$ structure, it gives a canonical $G^{+}(n,s^+,s^-)$ structure on $Y$. This is a one-to-one correspondence. 
\end{lemma}

\begin{proof}Let $\epsilon_1,\dots,\epsilon_{n+s^+}$ is a standard basis of $\R^n \oplus \R^{s^+}$ and $e'_0, e_0, e_1\dots e_{n+s^+ + s^-}$ are the generators of $Cl(0,(n+1)+(s^+ +1)+s^-)$. There is  an embeding $Cl(n+s^+, s^-) \to Cl(0,(n+1)+(s^+ +1)+s^-)$ given by mapping $\epsilon_1,\dots,\epsilon_{n+s^+}$ to $e'_0 e_0 e_1, \dots, e'_0 e_0 e_{n+s^+}$. Restricting this map, we have an embedding $G^+(n,s^+,s^-) \to Spin(1+n+s^+ +1+s^-)$. Let $\pi$ be the projection $Spin(1+n+s^+ +1+s^-) \to SO(1+n+s^+ +1+s^-)$. The image $\pi(G^+(n,s^+,s^-))$ is isomorphic to $S(O(n) \times O(s^+)\times O(s^-))$. This image gives the representation of $S(O(n) \times O(s^+)\times O(s^-))$ whose associated vector bundle is $\det TY \otimes \det E_+ \oplus TY \oplus E_+ \oplus \det E_- \oplus E_-$. Here we define that 
$\langle e_0 \rangle$, 
$ \langle e_1, \dots,e_n \rangle$, 
$\langle e_{n+1}, \dots, e_{n+s^+} \rangle$, 
$\langle e'_0 \rangle$ 
and $\langle e_{n+s^+ +1}, \dots,e_{n+s^+ +s^-}\rangle $ are the representations such that the associated vector bundle of the representations are 
$\det TY \otimes \det E_+$, 
$TY$, 
$E_+$, 
$\det E_-$ 
and $E_-$ respectively. Then we construct a $Spin$ structure as desired. We have the latter half of the statement by reversing the proof above. 
\end{proof}

We will prove some lemmas in a general setting. 
In the following three lemmas, let $M$ be a manifold and $F, E_0, E_1$ be oriented real vector bundles  on $M$ with fiber metrics. 

\begin{lemma}\label{spin0}
An orientation preserving isometry $\alpha \colon E_0 \to E_1$ determines a canonical $Spin$ structure on $E_0 \oplus E_1$.  
\end{lemma}
\begin{proof}
Let $n$ be the rank of $E_0$. The structure group of the subbundle 
\[
\{ (f,\alpha(f)) \mid f \text{ is an oriented orthonormal frame of $E_0$}\}
\]
 of the frame bundle of $E_0 \oplus E_1$ is a subgroup of $SO(2n)$ which is the image of the diagonal embedding $j \colon SO(n) \to SO(n) \times SO(n)  \hookrightarrow SO(2n)$. 
The map $j_* \colon \pi_1(SO(n)) \to \pi_1(SO(2n))$ is trivial, and therefore there exists a homomorphism $\tilde j \colon SO(n) \to Spin(2n)$ which covers $j$. 
Let$\{ g_{\alpha \beta}\}$ be the transition functions of $E_0 \oplus E_1$. We define the $Spin$ structure on $E_0 \oplus E_1$ by using $\{ \tilde{j}(g_{\alpha \beta})\}$ as transition functions. It is easy to check that this $Spin$ structure does not depend on how to take transition functions of $E_0 \oplus E_1$. 
\end{proof}

\begin{lemma}\label{spin}
Let $F, F'$ denote oriented real vector bundles on $M$ with metrics. If there are $Spin$ structures on $F \oplus F'$ and $F'$, these $Spin$ structures determine a canonical $Spin$ structure on $F$. 
\end{lemma}
\begin{proof}
We will denote by $\{ g_{\alpha \beta}\}, \{ g'_{\alpha \beta}\}$ transition functions of $F$ and $F'$ respectively. 
Let $\{ \tilde{g'}_{\alpha \beta}\}$ be the lifts of $\{ g'_{\alpha \beta}\}$ to the $Spin$ group defined by the  $Spin$ structure on $F'$. 
The lift of $(g_{\alpha \beta},  g'_{\alpha \beta})$ defined by the $Spin$ structure of $F \oplus F'$ is expressed by $[(\tilde{g}_{\alpha \beta},\tilde{g'}_{\alpha \beta})] \in Spin(n)\times Spin(m)/(\Z/2\Z)$ where $\tilde{g}_{\alpha \beta}$ is a lift of $g_{\alpha \beta}$. The lift $\tilde{g}_{\alpha \beta}$ is unique because the second component is fixed to be $\tilde{g'}_{\alpha \beta}$. 

The lift of the transition functions $\{ \tilde{g}_{\alpha \beta} \}$ satisfy the cocycle condition because 
$\{ [(\tilde{g}_{\alpha \beta},\tilde{g'}_{\alpha \beta})] \}$ and $\{ \tilde{g'}_{\alpha \beta} \}$ do. We define the $Spin$ structure on $F$ by using $\{ \tilde{j}(g_{\alpha \beta}\}) \}$ as transition functions. Again it is easy to check that this $Spin$ structure does not depend on how to take transition functions of $F$.
\end{proof}
The lemma below is clear from Lemma \ref{spin0} and Lemma \ref{spin}.  
\begin{lemma}\label{orispin}
Let $F \oplus E_0 \oplus E_1$ be a vector bundle with orientation $o$ and $Spin$ structure $\tilde{P}$. We assume that $E_1$ is oriented and there is an isometry $\phi \colon E_0 \to E_1$. We will denote by $o_E$ the orientation of $E_0$. 
Then $F$ has a natural orientation and a $Spin$ structure defined by $o, o_E, \tilde{P}$ and $\phi$. 
\end{lemma}

We use the lemmas above in our setting.  

\begin{definition}
Let $Y$ be an $n$ dimensional Riemannian manifold and $\mathfrak s =(\tilde{P},P,\pi,o,E_+, E_-)$ be a $G^{+}(n,s^+,s^-)$ structure on $Y$. We will denote by $h \in \Gamma(E_-)$ an transverse section. Let $C=h^{-1}(0)$. We call $C$ a caharacteristic submanifold of $\mathfrak s$.
\end{definition}

The theorem below is a generalization of the theorem that is a exsitence of a canonical $Spin$ structure on a characteristic submanifold of $Spin^c$ structure. 

\begin{theorem}\label{loc}
Let $Y$ be an $n$ dimensional Riemannian manifold and $\mathfrak s =(\tilde{P},P,\pi,o,E_+, E_-)$ be a $G^{+}(n,s^+,s^-)$ structure on $Y$ and $C$ be a characteristic submanifold of $\mathfrak s$. Then we have a natural $G^{+}(n-s^-, s^+, 0)$ structure on $C$.  
\end{theorem}

\begin{proof}We restrict the vector bundle
\[
\det TY \otimes \det E_+ \oplus TY \oplus E_+  \oplus \det E_- \oplus E_-
\]
to $C$. This is naturally isomorphic to the vector bundle
\[
\det TY|_C \otimes \det E_+|_C \oplus TC \oplus N \oplus E_+|_C \oplus \det E_-|_C \oplus E_-|_C
\]where $N$ is the normal bundle of $C$. 
We define the isomorphism $\psi \colon N \to E_-$ by using $dh$. 
By the definition of $G^{+}(n,s^+,s^-)$ structure, there exists an isomorphism $i \colon \det TY|_C \otimes \det E_+|_C \cong \det E_-|_C$. Let $F=\det TY|_C \otimes \det E_+|_C \oplus N \oplus \det E_-|_C \oplus E_-|_C,\; E_0=\det TY|_C \otimes \det E_+|_C \oplus N$, $E_1=\det E_-|_C \oplus E_-|_C$ and $\phi=i \oplus \psi$. They satisfy the assumption of Lemma \ref{orispin}, and hence we have the $Spin$ structure of $TC \oplus E_+$ as desired. From the latter statement of Lemma \ref{emb}, this $Spin$ structure defines the $G^+(n-s^-, s^+, 0)$ structure on $C$ as desired. 
\end{proof}

\begin{lemma}\label{bordism}
As an element of bordism group $\Omega_{n-s^-}^{G^+(n-s^-, s^+, 0)}(pt)$, $[(C,\mathfrak s _C (h))]$ is independent of $h$. 
\end{lemma}

\begin{proof}
Let $h, h'$ be transverse sections of $E_-$ and let $C=h^{-1}(0), C'=h^{-1}(0)$. We define a natural $G^+(n+1,s^+,s^-)$ structure on $Y\times [0,1]$ by using $G^+(n,s^+,s^-)$ structure on $Y$.  
 Let $\tilde h$ be a transverse section on $Y\times [0,1]$ such that $\tilde h |_{Y \times \{0\}}=h, \; \tilde h |_{Y \times \{1\}}=h'$. Then $\tilde h^{-1}(0)$ has a $G^+(n+1-s^-,s^+, 0)$ structure by Theorem \ref{loc}. This manifold with $G^+(n+1-s^-,s^+, 0)$ structure gives the cobordism from $[(C,\mathfrak s _C (h))]$ to $[(C',\mathfrak s _C' (h'))]$. 
\end{proof}

We now define the map $f$ of Theorem \ref{main}. 
\begin{definition}
Let $Y$ be an $n$ dimensional closed manifold with $G^{+}(n,s^+,s^-)$ structure $\mathfrak s =(\tilde{P},P,\pi,o,E_+, E_-)$. Set
\[f([Y,\mathfrak s])=[C,\mathfrak s_C].\] 
\end{definition}

\subsection{Localization by Witten deformation}\label{Witten_deformation}
In this subsection, we will prove the commutativity of the diagram in Theorem \ref{main}. We prove it by using Witten deformation. 
We prove some lemmas in a general setting rather than the setting in Theorem \ref{main}.

In the previous subsection, we define the $G^+(n-s^-,s^+,0)$ structure $\mathfrak s_C$ on a characteristic submanifold $C=h^{-1}(0)$ from a $G^+(n,s^+,s^-)$ structure $\mathfrak s$ on $Y$ in Theorem \ref{loc}. Now we consider the relation between the standard spinor bundle of $\mathfrak s_C$ and  the restriction of the standard spinor bundle $\Slash{\mathfrak S}$ of $\mathfrak s$ to $C$. First, we consider $\Slash{\mathfrak S} |_C$, the restriction of the standard $\Z/2\Z$ graded spinor bundle of $\mathfrak s$ to $C$.

\begin{proposition}\label{S_C}
Let $Y$ be an $n$ dimensional manifold with $G^+(n,s^+,s^-)$ structure $\mathfrak s$. 
We will denote by $\Slash{\mathfrak S}$ the standard $\Z/2\Z$ graded spinor bundle of $\mathfrak s$. Let $C=h^{-1}(0)$ and  let $\mathfrak s_C$ be the $G^+(n-s^-,s^+,0)$ structure on $C$ defined in Theorem \ref{loc}. We define a Clifford action on $\Slash{\mathfrak S} |_C$ by restricting the Clifford action $(c_{TY},c_{E_+}) \colon (TY \oplus E_+) \otimes \Slash{\mathfrak S} \to \Slash{\mathfrak S}$ to the vector bundle $(TC \oplus E_+ |_C) \otimes \Slash{\mathfrak S}|_C$ on $C$. We will denote by $(\tilde c_{TC}, \tilde c_{E_+})$ this restricted Clifford action. Note that the image of $(\tilde c_{TC}, \tilde c_{E_+})$ is $\Slash{\mathfrak S}|_C$. Then there exists an isomorphism $\Phi$ which has following properties: Let $\Slash{\mathfrak S} _C$ be the standard $\Z/2\Z$ graded spinor bundle of $\mathfrak s_C$. The isomorphism
\[
\Phi \colon \Slash{\mathfrak S} |_C \to Cl_+(N(C)) \hat{\otimes}\Slash{\mathfrak S} _C \hat{\otimes}V^{\ast}_{s^-}
\]
preserves the Clifford action of $TC \oplus E_+ |_C$ when we define a Clifford action to right hand side by
\begin{align*}
c_{TC}(x)\cdot \omega \otimes \phi \otimes v&:=(-1)^{\lvert \omega \rvert}\omega \otimes \tilde c_{TC}(x)\phi \otimes v, \\
c_{E_+}(w)\cdot \omega \otimes \phi \otimes v&:=(-1)^{\lvert \omega \rvert}\omega \otimes \tilde c_{E_+}(w)\phi \otimes v
\end{align*} for $\omega \otimes \phi \otimes v \in Cl_+(N(C)) \hat{\otimes}\Slash{\mathfrak S} _C \hat{\otimes}V^{\ast}_{s^-},\; x \in TC$ and $w \in E_+ |_C$. Moreover, $\Phi$ preserves the right $Cl_{(n+s^+, s^-)}$ action if we define the right $Cl_{(n+s^+, s^-)}$ action to $Cl_+(N(C)) \hat{\otimes}\Slash{\mathfrak S}_C \hat{\otimes}V^{\ast}_{s^-}$ by 
\[
\omega \otimes \phi \otimes v \cdot a \otimes b :=(-1)^{\lvert a \rvert\lvert v \rvert} \omega \otimes \phi \cdot a \otimes vb\] for $\omega \otimes \phi \otimes v \in Cl_+(N(C)) \hat{\otimes}\Slash{\mathfrak S} _C \hat{\otimes}V^{\ast}_{s^-}$ and $a\otimes b \in Cl_{(n+s^+ -s^-,0)} \hat{\otimes} Cl_{(s^-, s^-)}$. 
\end{proposition}

\begin{proof}
From the proof of Theorem \ref{loc}, we see the structure group of the $Spin$ structure of $\det N(C) \oplus \det E_- \oplus N(C) \oplus E_-$ reduces the subgroup of $Spin(2s^- +2)$ which is isomorphic to $O(s^-)$. If we write $O(s^-)$ for the subgroup, then the structure group of $\Slash{\mathfrak S}|_C$  is $G^+(n-s^-, s^+,0) \times O(s^-)$. 
Write $G=G^+(n-s^-, s^+,0) \times O(s^-)$. 

Let us denote by $S=Cl_{(n+s^+, s^-)}$ a representation space whose associated vector bundle is $\Slash{\mathfrak S}$. We describe the action of $G$ on $S$. We define $\tilde g_v \in O(s^-)$ to be the natural lift of $g_v$ to $Spin(2s^- +2)$ given in Lemma \ref{spin0} where $g_v$ is the composition of the reflections of vectors$(e,0, e,0), (0,v, 0,v) \in \det \R^{s^-} \oplus \R^{s^-} \oplus \det \R^{s^-} \oplus \R^{s^-}$ $(\lvert e \rvert=\lvert v \rvert=1)$. The element $g_v$ is independent of the choice of $e$. Elements in the group $O(s^-) \subset Spin(2s^- +2)$ are the products of finitely many elements of the form $\tilde g_v$ for some $v$ and $\pm 1$. The Clifford multiplications $c'_n, c'_-$ are $G$ equivariant and hence the action of $\tilde g_v$ is given by $\pm c'_n(v)c'_-(v)$. 

The sign of $\tilde g_v =\pm c'_n(v)c'_-(v)$ is determined by how to embed the group $G^+(n, s^+, s^-)$ to a $Spin$ group. In our convention of Lemma \ref{emb}, the image of $c'_n(e_i)c'_-(e_i)$ of the embedding is $e_0 e_i e'_0 e'_i$ and this coincides with $\tilde g_{e_i}$ which is a lift of $g_v$ given in Lemma \ref{spin0}. Thus we have $\tilde g_v =c'_n(v)c'_-(v)$. 

We will denote by $S_C$ a subspace of $S$ which is the intersection of $+1$-eigenspaces of $\tilde g_v=c'_n(v)c'_-(v)$ for all $v \in S(\R^{s^-})$. 
The subspace $S_C$ coincides with the intersections of $+1$ eigenspaces of $c'_n(e_i)c'_-(e_i)$($i=1, \dots,s^-$) where $\{ e_1, \dots, e_{s^-} \}$ is an orthonormal basis of $\R^{s^-}$. Note that the action of $G^+(n-s^-, s^+, 0)$ preserves $S_C$ because this action commutes with $c'_n(v)c'_-(v)$ for all $v$. The definition of $S_C$ immediately implies that the action of $O(s^-)$ on $S_C$ is trivial. 

Let us show that $S_C$ coinsides with $Cl_{(n+s^+ -s^-, 0)} \subset S$. Recall that the standard spinor bundle $\Slash{\mathfrak S}$ is the associated bundle with the representation $Cl_{(n+s^+,s^-)}=Cl_{(n-s^- +s^+,0)} \hat{\otimes} Cl_{(s^-,s^-)}$. An element $\tilde g_v = c'_n(v)c'_-(v)$ acts on $Cl_{(n-s^- +s^+,0)}$ trivially and on $Cl_{(s^-,s^-)}$ by left multiplication of $Cl_{(s^-,s^-)}$. We will denote by $\tilde V$ the intersection of the $+1$-eigenspaces of $c'_n(v)c'_-(v)$ in $Cl_{(s^-,s^-)}$ for all $v \in S(\R^{s^-})$. The dimension of $\tilde V$ is $2^{s^-}$. The action of $c'_n(v)c'_-(v)$ commutes with the right $Cl_{(s^-,s^-)}$ action. The only irreducible representation of $Cl_{(s^-,s^-)}$ is $V_{s^-}$ and its dimension is $2^{s^-}$. Hence there is a isomorphism $\psi \colon V^{\ast}_{s^-} \to \tilde V$. Using $\psi$, we have the isomorphism $Cl_{(n+ s^+ -s^-, 0)} \hat{\otimes} V^*_{s^-} \to S_C$ by $\xi \otimes v \mapsto \xi \psi(v)$. If we give the $G$ action on $Cl_{(s^-,0)}$ by $g \cdot x \mapsto gxg^{-1}$, we have the $G$ equivariant isomorphism
\[
\begin{array}{rccc}
\Psi_0 \colon& Cl_{(s^-,0)} \hat{\otimes}Cl_{(n+ s^+ -s^-, 0)} \hat{\otimes} V^*_{s^-}& \longrightarrow &Cl_{(n+s^+,s^-)} \\
&\rotatebox{90}{$\in$} & & \rotatebox{90}{$\in$} \\
&x \otimes \xi \otimes v & \longmapsto & x \cdot \xi \cdot \psi(v)
\end{array}.
\]

The associated vector bundles of $Cl_{(s^-,0)}, Cl_{(n+ s^+ -s^-, 0)}$ and $V^*_{s^-}$ is $Cl_{+}(N(C)), \Slash{\mathfrak S}_C$ and the trivial bundle $V^*_{s^-}$ respectively. Thus we have the map
\[
\Psi \colon Cl_+(N(C)) \hat{\otimes}\Slash{\mathfrak S}_C \hat{\otimes} V^*_{s^-} \longrightarrow \Slash{\mathfrak S}|_C .
\]
We define the map $\Phi$ to be the inverse of the map $\Psi$. It is easy to check to see that $\Phi$ preserves the Clifford action of $TC \otimes E_+|_C$ and right $Cl_{(n+s^+, s^-)}$ action. 
\end{proof}


In this section we identifies $\Slash{\mathfrak S}|_C$ with $ Cl_+(N(C)) \hat{\otimes}\Slash{\mathfrak S}_C \hat{\otimes} V^*_{s^-}$ by the map $\Phi$. 

Next we prove the localization of analytic index in Proposition \ref{Wittendeformation}. 
We identify a tubular neighborhood of $C$, denoted by $U(C)$, with the open disk bundle of the normal bundle of $C$,  $B(N(C))=\{v \in N(C) \mid \lvert v \rvert <1 \}$. 
We will denote by $\pi \colon N(C) \to C$ the projection of the normal bundle. 
Let us outline the proof of the localization. 
\begin{itemize}
\item[Step.1] We first formulate the index of a Dirac type operator acting on the sections of the vector bundle $\pi^* (\Slash{\mathfrak S} |_C )$ over $N(C)$. Since $N(C)$ is not closed, we need behavior of its end. We will see the index of $N(C)$ coincides with the index of the $G^+(n,s^-,s^+)$ structure $\mathfrak s_C$ on $C$. 
\item[Step.2] The analytic index of the $G^+(n,s^+,s^-)$ structure $\mathfrak s$ on $Y$ coincides with the index of Step 1. In this argument we deform the Dirac type operator in Definition\ref{defineindex}. We give prove some analytic lemmas and propositions of this step in the Appendix. 
\end{itemize}

\subsubsection{Step.1}

We begin with Step 1. First, we consider the trivial $G^+(n,0,n)$ structure on $\R^n$. This structure appears in a fiberwise way when we consider the operator on $N(C)$. Let $\mathfrak s_0$ be a trivial $G^+(n,0,n)$ structure on $\R^n$. We will denote by $h(x)=x$ the section of $E_-=\R^n \times \R^n$ and we have the isomorphism from $T\R^n$ to $E_-$ by using $dh$. We have the reduction of the structure group of $\mathfrak s_0$ to the subgroup $G \subset G^-(n,0,n)$ which is naturally isomorphic to $O(n)$ by using Lemma \ref{spin0}. Let $\Slash{\mathfrak S}_f$ denote the standard $\Z/2$ graded spinor bundle of $\mathfrak s_0$. Let $C=h^{-1}(0)=\{0\}$, be a characteristic submanifold. From the Proposition \ref{S_C}, we identify $\Slash{\mathfrak S}_f|_C$ with $Cl_+(\R^n)\hat{\otimes}\underline{\R}\hat{\otimes}V^{\ast}_n$. Note that $\Slash{\mathfrak S}_C \cong C \times \R=\underline{\R}$ is a vector bundle on a point. Let $L$ be a trivial vector bundle on $\R^n$. 

\begin{lemma}\label{fiber_harmonic}
We will denote by $\mathcal S$ a set of rapidly decreasing sections of $\Slash{\mathfrak S}_f\otimes L$. 
We will denote by $D'_m$ a differential operator acting on sections of $\Gamma(\Slash{\mathfrak S}_f \otimes L)$ given by
\[D'_m=\sum_{i=1}^{n}\left( c_{T\R^n}(dx^i)\frac{\partial}{\partial x^i}+mc_{E_-}(e_i)x^i \right).\]
The operator $D'_m$ is independent of the choice of the orthonormal basis of $\R^n$. This operator commutes with the right $Cl_{(n,n)}$ action, and there is an isomorphism between $\ker D'_m \cap \mathcal S$ and  $\underline{\R}\hat{\otimes}V^{\ast}_n\otimes L$ as right $Cl_{(n,n)}$ modules. 
\end{lemma}
\begin{proof}
We see at once that $D'_m$ is independent of choice of the orthonormal basis of $\R^n$ by direct calculation. 

Let $\phi$ be a rapidly decreasing section. We have $D'_m \phi=0$ if and only if $(D'_m)^2 \phi=0$ because $D'_m$ is a skew-symmetric operator. We have
\begin{align*}
(D'_m)^2&=\sum_{i=1}^n \left( \left( \frac{\partial}{\partial x^i} \right)^2 -m^2(x^i)^2 + m c_{T\R^n}(dx^i)c_{E_-}(e_i)\right) \\
&=-H + \sum_{i=1}^n \left(m c_{T\R^n}(dx^i)c_{E_-}(e_i)\right)
\end{align*}
 where $H$ is a harmonic oscillator acting on smooth functions $\R^n \to \R$. 
It is well known that $H$ has only discrete spectrum and each eigenspace is $1$-dimensional. The eigenvalues are $nm, 2nm, 3nm, \dots$ and each eigenfunction is rapidly decreasing. In particular, the eigenfunction of $nm$ is $e^{-m\lvert x \rvert^2 /2}$ and this does not depend on the choice of orthonormal basis of $\R^n$. The eigenvalues of $m c_{T\R^n}(dx^i)c_{E_-}(e_i)$ are $\pm m$ for all $i$. Hence the kernel of $(D'_m)^2$ is in the intersection of the $+1$ eigenspace of $c_{T\R^n}(dx^i)c_{E_-}(e_i)$ and $nm$ eigenspace of $H$. The intersection of the $+1$ eigenspace of $c_{T\R^n}(dx^i)c_{E_-}(e_i)$ for all $i$ coinsides with $\underline{\R}\hat{\otimes}V^{\ast}_n\otimes L$ from the constructuion of the map $\Phi$ on Proposition\ref{S_C}. Thus we have $\ker D'_m \cap \mathcal S=\underline{\R}\hat{\otimes}V^{\ast}_n\otimes L \cdot e^{-m\lvert x \rvert^2 /2}$.  
\end{proof}

\begin{definition}\label{u0}
We will denote by $\tilde u _0$ a function on $\R^n$ given by 
\[\frac{ e^{-m\lvert x \rvert^2 /2}}{\lVert e^{-m\lvert x \rvert^2 /2} \rVert_{L^2(\R^n)} }.\]
\end{definition}

\begin{definition}\label{E_G+} 
We will follow the notation of the statement and the proof of Proposition \ref{S_C}. We deform a metric of $Y$ if necessary and we identify a tubular neighborhood $U(C)$ of $C$ with $B(N(C))=\{v \in N(C) \mid \lvert v \rvert <1 \}$ as a Riemannian manifold. Let us denote by $\pi \colon N(C) \to C$ the projection. We define a $G^+(n,s^+,s^-)$ structure on $B(N(C))$ by restricting the $G^+(n,s^+,s^-)$ structure $\mathfrak s$ on $Y$. By abuse of notation, we use the same letter $\mathfrak s$ for this $G^+(n,s^+,s^-)$ structure on $B(N(C))$ and we write $\Slash{\mathfrak S}$ instead of $\Slash{\mathfrak S}|_{U(C)}$. The structure group of $\mathfrak s$ reduces to $G=G^+(n-s^-, s^+,0) \times O(s^-)$ by using the isomorphism $dh \colon N(C) \cong E_-$. Let $\Slash{D}_m=\Slash{D}+mc_-(h)$ where $\Slash{D}$ is a Dirac operator of $\Slash{\mathfrak S}$. We abbreviate $c_-(h)$ to $h$. Note that $\Slash{D}_m$ is an anti-symmetric operator. 
Let $A$ be a connection such that $c_{TY} \circ A= \Slash D$. Perturbing $A$, we may asuume that $\pi^* (A|_C)=A$ on $B(N(C))$. 
\end{definition}

\begin{lemma}\label{D'}
Let $\Slash{D}'_C$ be the Dirac operator defined by the Clifford action $c_{TC} \colon TC \otimes \Slash{\mathfrak S}|_C \to \Slash{\mathfrak S}|_C $. Then we have a decomposition
\[
\Slash D = D_C + D_f
\]
with the following property:  
\begin{itemize}
\item On $U \times \R^{s^-}$, which is a trivialization of $N(C)$, 
\[
D_C = \Slash{D}'_C |_U, \; D_f = \sum_{k=1}^{s^-}c_{TY}(d\xi_i)\frac{\partial}{\partial \xi_i}. 
\]
Here $c_{TY}$ is the restriction of the Clifford multiplication of $\mathfrak s$ to $\pi^* N(C) \subset TN(C)$ and 
$(\xi_1, \dots,\xi_{s^-})$ are coordinates of $\R^{s^-}$, which is a fiber of $N(C)$. 

\item $D_C$ and $D_f$ are anti-commutes. 
\item Both $D_C$ and $D_f$ are anti-symmetric with respect to the $L^2$ inner product on $N(C)$. 
\end{itemize}
\end{lemma}
\begin{proof}
It is easy to see that we have the decomposition of $\Slash D$ on each trivialization of $N(C)$. The operator
\[
\sum_{k=1}^{s^-}c_{TY}(d\xi_i)\frac{\partial}{\partial \xi_i}
\]
 is $O(n)$ invariant. Hence the operator above on each trivialization coincides with the corresponding one on another trivialization. Thus we have the operator $D_f$ on the whole of $N(C)$. Let $D_C:=\Slash D - D_f$. On each trivialization, the operator $D_f$ coincides with $\sum_{k=1}^{s^-}c_{TY}(d\xi_i)\frac{\partial}{\partial \xi_i}$ and $D_C$ coincides with $\Slash{D}'_C |_U$. We will see at once anti-commutativity of these operators on each trivialization on $N(C)$. 
We see these operators are anti-symmetric with respect to inner product $L^2$. It is sufficient to prove that $D_f$ is anti-symmetric because the operator $\Slash D$ is anti-symmetric. To see this, we decompose  an integral on $N(C)$ by $\int_{N(C)}=\int_{C} \int_{\text{fiber}}$ where $\int_{\text{fiber}}$ is the integration on each fiber. The operator $D_f$ is anti-symmetric with respect to the $L^2$ inner product on each fiber and we have $D_f$ is anti-symmetric  with respect to the $L^2$ norm on $N(C)$.  
\end{proof}

\begin{definition}\label{u_0}
We will denote by $u_0$ the function on $N(C)$ whose restriction to each fiber of $N(C)$ coincides with the function $\tilde u _0$. 
\end{definition}

\begin{lemma}\label{D_C}
We will denote by $h \in \Gamma(\pi^* E_-)$ the tautological section of $\pi^* E_- \cong \pi^* N(C)$. Let $\Slash{D}_C$ be a Dirac operator on $\Slash{\mathfrak S}_C\hat{\otimes}V^{\ast}_{s^-}$. Let $m$ be a positive real number and let $a\in \Gamma(\Slash{\mathfrak S}_C\hat{\otimes}V^{\ast}_{s^-})$. We have 
\[
(\Slash D + mh ) \pi^* a \cdot u_0 =\pi^*(\Slash{D}_C a) \cdot u_0.
\]
\end{lemma}
\begin{proof}
Recall the decomposition $\Slash D = D_C + D_f$. The restriction of $D_f + mh$ to each trivialization of $N(C)$ coincides with the operator $D'_m$ in Lemma \ref{fiber_harmonic}. 
Thus we have $(D_f + mh)(\pi^* a \cdot u_0)=0$. It is sufficient to consider the operator $D_C$. The operator $D_C$ coincides with the operator $\Slash{D}'_C$ on each trivialization and this operator is a pullback of an operator on $C$. Hence we have 
$
D_C ( \pi^* a \cdot u_0 )=\pi^*(\Slash{D}_C a) \cdot u_0. 
$
\end{proof}

\begin{lemma}\label{eigenvalue}
Let $H=\{  \pi^* a \cdot u_0 \in \Gamma(\Slash{\mathfrak S})\mid a \in \Gamma(\Slash{\mathfrak S}_C\hat{\otimes}V^{\ast}_{s^-}), \Slash{D}_C a = 0 \}$. 
We will denote by $\lambda_C$ the minimum of the absolute value of non-zero eigenvalues of $\Slash{D}_C$. Let $m$ be a positive number such that $m \ge \lvert \lambda_C \rvert $. 
Let $\phi \in \Gamma(N(C), \Slash{\mathfrak S})$ be a section whose restriction to each fiber of $N(C)$ is rapidly decreasing function. If $\phi$ is orthogonal to $H$ in the $L^2$ inner product, We have $\lVert \Slash D_m \phi \rVert \ge \lambda_C \lVert \phi \rVert$ where $\lVert \cdot \rVert$ is the $L^2$ norm on $N(C)$. 

\end{lemma}
\begin{proof}
We decompose $\phi$ into $\phi=\phi_0 +  \pi^* b \cdot u_0$ where $b \in \Gamma(\Slash{ \mathfrak S}_C)\hat{\otimes}V^{\ast}_{s^-}$ and $\phi_0$ is a fiberwise rapidly decreasing section such that $\int_{N(C)_{x}} \langle \phi_0, \pi^* a \cdot u_0 \rangle =0$ for all $x \in C, \; a \in (\Slash{\mathfrak S}_C\hat{\otimes}V^{\ast}_{s^-})_{x}$. From Lemma \ref{fiber_harmonic} and $L$ to be $(\Slash{\mathfrak S}_C)_x$, $\phi_0$ is orthogonal to the kernel of $D_f + mh$ on each fiber. The section $b$ is orthogonal to $\ker \Slash D _C$ in the $L^2$ inner product on $C$ because $\phi$ is orthogonal to $H$. 
From the above decomposition, we have
\[
(\Slash D + mh )\phi = (\Slash D + mh )\phi_0 + \pi^*(\Slash{D}_C b) \cdot u_0.
\]
It is easy to see that $\int_{N(C)} \langle (\Slash D + mh )\phi_0, \pi^*(\Slash{D}_C b) \cdot u_0 \rangle=0$ and $\lVert (\Slash D + mh)\phi \rVert^2 =\lVert (\Slash D + mh )\phi_0 \rVert^2+ \lVert \pi^*(\Slash{D}_C b) \cdot u_0 \rVert^2$. It is straightforward to see $D_C$ and $D_f + mh$ are anti-commutative.

Thus we have
\begin{align*}
\int_{N(C)} \lvert (\Slash D + mh )\phi_0\rvert^2 
&= \int_{N(C)}\lvert D_C \phi_0\rvert^2 + \int_{N(C)} \lvert (D_f + mh )\phi_0\rvert^2 \\ 
&\ge \int_{C} \int_{\text{fiber}} \lvert (D_f + mh )\phi_0\rvert^2 \\
&\ge \int_{C} m^2 \int_{\text{fiber}} \lvert \phi_0\rvert^2 \\
&\ge \int_{C} \lambda_C^2 \int_{\text{fiber}} \lvert \phi_0\rvert^2
\end{align*}

and
\begin{align*}
\int_{N(C)} \lvert \pi^*(\Slash{D}_C b) \cdot u_0  \rvert^2
&= \int_{N(C)} \lvert \pi^*(\Slash{D}_C b) \cdot u_0  \rvert^2 \\
& \ge \lambda_C^2 \int_{N(C)} \lvert \pi^* b \cdot u_0  \rvert^2. 
\end{align*}
Hence we have the inequality $\lVert \Slash D_m \phi \rVert \ge \lambda_C \lVert \phi \rVert$ as desired. 
\end{proof}

\begin{definition}
The finite-dimensional vector space $H$ has a natural $\Z/2\Z$ grading and a left $Cl_{(n+s^+, s^-)}$ action induced by $\Slash{\mathfrak S}$. Let $\epsilon$ be the $\Z/2\Z$ grading operator. The four tuples $(0, \epsilon,Cl_{(n+s^+,s^-)}, H)$ defines an element of $KO^{s^- -n-s^+}(pt)$. We write $\ind (N(C), \Slash{\mathfrak S})$ for this element in $KO^{s^- -n-s^+}(pt)$. 
\end{definition}

\begin{remark}\label{indNC=indC}
By the definition of $H$ and from Lemma\ref{H=H'V}, we have \[\ind(\mathfrak s_C)=[(\Slash{D}_C, \epsilon, Cl_{(n+s^+,s^-)}, L^2(C, \Slash{\mathfrak S}_C\hat{\otimes}V^{\ast}_{s^-}))]=\ind(N(C), \Slash{\mathfrak S}) \in KO^{s^- -n-s^+}(pt).\]
\end{remark}

\subsubsection{Step.2}
 
In this step, we prove $\ind(\mathfrak s)$, an analytic index of $G^+(n,s^+,s^-)$ structure on $Y$, is equal to $\ind(N(C), \Slash{\mathfrak S})$. 

We introduce some notations. 
\begin{itemize}
\item If necessary we perturb $h$ so that $h$ satisfies $\lvert h \rvert < 1$ on $U(C)$ and $\lvert h \rvert = 1$ on the complement set of $U(C)$. 
\item We will denote by $\mathcal{H}_{m}^{\lambda}$ the direct sum of eigenspaces of $\Slash{D}_m \colon \Gamma(Y,\Slash{\mathfrak S}) \to \Gamma(Y,\Slash{\mathfrak S})$ such that the absolute value of each eigenvalue is less than $\lambda^2$. This is a finite-dimensional subspace of $L^2(Y, \Slash{\mathfrak S})$. $\mathcal{H}_{m}^{\lambda}$ has a natural $\Z/2\Z$ grading and a left $Cl_{(n+s^+,s^-)}$ action. 
\item Let $\rho$ be a smooth function on $Y$ supported in $U(C)$ such that it satisfies  
      \[
      \rho(z)=\begin{cases}
1 & \text{if $ \lvert z \rvert \le 1/2 $ } \\
0 & \text{if $\lvert z \rvert \ge 2/3$}
\end{cases}
      \]
      for $z \in U(C) \cong B(N(C)) \subset N(C)$ and monotone decreasing on $1/2 < \lvert z \rvert <2/3$ in $\lvert z \rvert$ where $\lvert \cdot \rvert$ is a norm of $N(C)$. (We identify $U(C) \cong B(N(C))$.) 
\end{itemize}

The following proposition is proved by the general theory of Witten deformation. We prove this proposition in the section of Appendix. 

\begin{proposition}\label{Wittendeformation}
Assume that a positive constant $\lambda $ is smaller than a constant determined by the principal symbol of $D$ and the differential of $\rho$ and suppose that $m >\lambda$. 
Let $\Pi'$ be the orthogonal projection from $L^2(N(C), \Slash{\mathfrak S})$ to $H$. 
Then the map
\[
\begin{array}{r@{\,\,}c@{\,\,}c@{\,\,}c}
&\mathcal{H}_m^{\lambda} &\longrightarrow & H \\
&\rotatebox{90}{$\in$}&&\rotatebox{90}{$\in$}\\
& \phi &\longmapsto & \Pi'(\rho \phi)
\end{array}
\]
is isomorphic and preserving the $\Z/2$ grading and the $Cl_{(n+s^+,s^-)}$ action.
\end{proposition}

Now we prove the main theorem. 
\begin{proof}[Proof of Theorem \ref{main}] 
From the definition of the index of $\mathfrak s$, \[ \ind(\mathfrak s)=[(\Slash{D}_m, \epsilon, Cl_{(s^-,s^+ +n)}, \mathcal{H}_m^{\lambda})]\] and 
we see \[
[(0,\epsilon,Cl_{(s^-,s^+ +n)}, H)]=[(\Slash{D}_m, \epsilon, Cl_{(s^-,s^+ +n)}, \mathcal{H}_m^{\lambda})] \in KO^{s^- -n- s^+}(pt)\] from Proposition \ref{Wittendeformation}. We have $[(0,\epsilon,Cl_{(s^-,s^+ +n)}, H)]=\ind ({\mathfrak s}_C)$ from Remark \ref{indNC=indC}. Thus we have proved Theorem \ref{main}. 
\end{proof}

\section{Examples}
\subsection{Freed-Hopkins' $H_n(s)$}\label{eg}
In this subsection, we consider a family of groups $H_n(s)$ defined by Freed and Hopkins \cite{FreedHopkins1}. 
The group $H_n(s)$ is given by the table below. 

\begin{center}
\begin{tabular}{lrr} \toprule
$s$ & $H_n (s)$   \\ \midrule
$0$ & $Spin(n)$ \\
$-1$ & $Pin^+(n)$ \\ 
$-2$ & $Pin^+(n) \ltimes_{\Z/2\Z} U(1)$ \\
$-3$ & $Pin^-(n) \times_{\Z/2\Z}SU(2)$ \\
$4$ & $Spin(n) \times_{\Z/2\Z}SU(2)$ \\
$3$ & $Pin^+(n) \times_{\Z/2\Z}SU(2)$ \\
$2$ & $Pin^-(n) \ltimes_{\Z/2\Z} U(1)$ \\
$1$ & $Pin^-(n)$ \\
\bottomrule
\end{tabular}
\end{center}

\begin{lemma}\label{H_n(s)}
In the case $s=0,1,2,3$, we have $H_n(s) \cong G^+(n,s,0)$ and in the case $s=-1,-2,-3$, we have $H_n(s) \cong G^+(n,0,-s)$.  
\end{lemma}
\begin{proof}
It is obvious for the case $s=0$ therefore it is sufficient to prove the case of $\lvert s \rvert=1,2,3$. First we consider the case $\lvert s \rvert=3$. We use the identification $Spin(3) \cong SU(2)$ and $(Cl_{+n}\hat{\otimes}(\R \oplus \R \Gamma))_0 \cong Cl_{\pm n} $ where $\Gamma=e_1 e_2 e_3 \in Cl_s$ and $\pm$ is the sign of $s$. 
We have an isomorphism $G^+(n,s,0) (\text{or}\; G^+(n,0,-s)) \to H_n(s)$ given by 
\[
gu \mapsto
\begin{cases}
[g,u] & g\in Spin(n), u \in Spin(3) \\
[g\otimes \Gamma,\Gamma^{-1} u] & \text{otherwise}
\end{cases}
\]
for $g \in Pin^+(n), u \in Pin^{\pm}(3)$. 

We consider the case $\lvert s \rvert=1,2$. We will denote by $e$ a generator of $Cl_{s}$ with $\lvert e\rvert=1$. We use the identification $Cl_{+n}\hat{\otimes}(\R \oplus \R e) \cong Cl_{\pm n}$ where $\pm$ is the sign of $s$. 
We have an isomorphism $G^+(n,s,0) (\text{or} G^+(n,0,-s)) \to H_n(s)$ given by
\[
gu \mapsto
\begin{cases}
[g,u] & g\in Spin(n), u \in Spin(\lvert s \rvert) \\
[g\otimes e,e^{-1} u] & \text{otherwise}
\end{cases}
\]
for $g \in Pin^+(n), u \in Pin^{\pm}(\lvert s \rvert)$. 
\end{proof}
\begin{remark}
An index of $H_n(s)$ structure is defined in \cite{FreedHopkins1}. From the lemma above, we see the index of $H_n(s)$ coincides with our index of $G^+(n,s,0)\;(\text{or}\; G^+(n,0,-s))$ structure. 
\end{remark}

\begin{remark}
In the case $s=4$, $H_n(4)$ is isomorphic to a subgroup of $G^+(n,0,4)$ and we can see the index of $H_n(4)$ structure defined in \cite{FreedHopkins1} coincides with our index if the  structure group of a $G^+(n,0,4)$ structure reduces to $H_n(4)$. 

We identify $Spin(4)$ with $SU(2) \times SU(2)$. We give an embedding of $H_n(4)$ to $G^+(n,0,4)$ by
\[
Spin(n) \times_{\Z/2\Z}SU(2) \ni [g,u] \mapsto [g, \text{diag}(u,u)] \in Spin(n) \times_{\Z/2\Z}Spin(4) \subset G^+(n,0,4).
\]
\end{remark}

The following proposition is a consequence of Theorem \ref{main}, Lemma \ref{H_n(s)} and the above remarks.  
\begin{proposition}
We assume $s=-1,-2,-3,$ or $4$. Then there exists an isomorphism $f$ such that the following diagram commutes. 
\[
\begin{diagram}
\node{\Omega_{n}^{H_n(s)}(pt)}\arrow{e,t}{f} \arrow{se}\node{\Omega_{n-\lvert s\rvert}^{Spin}(pt)}\arrow{s} \\
\node[2]{KO^{-n-s}(pt)} 
\end{diagram}
\]
\end{proposition}

\subsection{The case of $G^+(5,0,4)$ structure}\label{G^+(5,0,4)}
We consider the index of the $G^+(5,0,4)$ structure especially because it is important to see the orientability of $Pin^-(2)$ monopole moduli space in the next section. The following arguments can be easily generalized to the case of the index of the $G^+(8k+5,0,4)$ structure. We use only the $k=0$ case in the next section therefore we only consider this case to avoid complications. 

\begin{definition}
Let $Z$ be a $5$ dimensional closed manifold and $\mathfrak s$ be a $G^+(5,0,4)$ structure on $Z$. Let $\Slash{\mathfrak S}$ be the standard spinor bundle of $\mathfrak s$. Note that $\Slash{\mathfrak S}$ has a natural right $Cl_{(5,4)}$ action. Let us denote by $\epsilon_0,\epsilon_1,\dots,\epsilon_4$ the generators of $Cl_{+5}$ and by $e_1, \dots, e_4$ the generators of $Cl_{-4}$. We will denote by $\tilde S$ a subbundle of $\Slash{\mathfrak S}$ such that the intersection of $+1$-eigenspaces of $\epsilon_{1} e_1,\dots, \epsilon_{4} e_4$. We call $\tilde S$ the spinor bundle of $\mathfrak s$. 
The Clifford multiplication of $TZ$ and $E_-$ preserves $\tilde S$ because they commute with the right $Cl_{(5,4)}$ action. 
We define a skew-adjoint Dirac operator $\tilde D$ on $\tilde S$ by using Clifford action of $TZ$. 
\end{definition}

We have the following lemma from Lemma \ref{H=H'V} :

\begin{lemma}
The spinor bundle $\tilde S$ is a generalized $\Z/2$ graded spinor bundle with left $Cl_{-1}$ action. The index $\ind({\mathfrak s})$ of $\mathfrak s$ coincides with $(0,\epsilon, Cl_{-1},\ker \tilde D) \in KO^{-1}(pt)$. Moreover, under the isomorphism $KO^{-1}(pt) \cong \Z/2\Z$, we have $\ind(\mathfrak s)=\dim \ker (\tilde D|_{\tilde{S}^+}) \mod 2$ where $\tilde{S}^+$is the even part of $\tilde S$. 
\end{lemma}

We give the explicit definition of $\tilde S$ in the case that structure group of $\mathfrak s$. This definition is useful to consider the orientability of $Pin^{-}(2)$ monopole moduli space.  

\begin{definition} \label{basis}
We will denote by $\epsilon_0, \dots, \epsilon_4$ the generators of $Cl_{+5}$ and by $e_1, \dots,e_4$ the generators of $Cl_{-4}$. 
\begin{itemize}
\item Let $\q$ be the quantenion ring. We use the conventions that $ijk=1$. 
\item Let $\q(2)$ be the matrix ring of $2 \times 2$ matries with entries in the ring $\q$. 
\item We define an isomorphism $\alpha \colon Cl_{-4} \to \q(2)$ by
\begin{align*}
 e_1 \mapsto \begin{pmatrix}0&1\\-1&0\end{pmatrix}, \; e_2 \mapsto \begin{pmatrix}0&i\\i&0\end{pmatrix} \;
 e_3 \mapsto \begin{pmatrix}0&j\\j&0\end{pmatrix}, \; e_4 \mapsto \begin{pmatrix}0&k\\k&0\end{pmatrix}.
\end{align*}
\item Let $\Gamma=e_1e_2e_3e_4 \in Cl_{-4}$. Note that $\Gamma^2=1$ and \[\alpha(\Gamma)=\begin{pmatrix}1&0\\0&-1\end{pmatrix}.\]
\item We define an isomorphism $\beta \colon Cl_{+5} \to \q(2) \oplus \q(2)$ by
\[
\epsilon_0 \mapsto (\alpha(\Gamma),-\alpha(\Gamma)), \; \epsilon_i \mapsto (\alpha(\Gamma)\alpha(e_i), -\alpha(\Gamma)\alpha(e_i)), \; i=1,2,3,4.
\]
\item We define an isomorphism $f \colon Cl_{(5,4)} \to (\q(2) \oplus \q(2))\otimes \q(2)$ by 
\[
\epsilon_p \mapsto \beta(\epsilon_p)\otimes \alpha(\Gamma), \; e_q\mapsto 1 \otimes \alpha(e_q)
\]
where $p=0,\dots,4, \; q=1, \dots,4$. 
\item  Using isomorphisms $\alpha, \beta$, we define the spinor representations of $Cl_{+5}, Cl_{-4}$ whose representation spaces are $\q^2 \oplus \q^2, \; \q^2$ respectively. By abuse of notation, we let $\alpha, \beta$ stand for these spinor representations respectively. Note that $\beta(\epsilon_i) \cdot (\phi_0, \phi_1)=(\alpha(\Gamma)\alpha(e_i)\phi_0, -\alpha(\Gamma)\alpha(e_i)\phi_1)$ for $i=0, \dots 4$. 
\end{itemize}
\end{definition}
\begin{remark}
In our notations, the $\Z/2\Z$ grading of $\q(2) \oplus \q(2)$ induced by the isomorphism $\beta$ is given as follows: The subspace $\{ (A, A) \mid A \in \q(2) \}$ is the even part and $\{ (A, -A) \mid A \in \q(2) \}$ is the odd part. The $\Z/2\Z$ grading operator is a map $(A,B) \mapsto (B,A)$. 

The $\Z/2\Z$ grading of $\q(2)$ induced by the $\Z/2\Z$ grading of $Cl_{-4}$ using the map $\alpha$ is given by declearing the even part consists of diagonal matrices and the odd part of off-diagonal one. 
The $\Z/2\Z$ grading operator is given by $A \mapsto \alpha(\Gamma)A\alpha(\Gamma)^{-1}=Ad(\alpha(\Gamma))(A)$. 
\end{remark}

\begin{definition} 
We assume that the structure group of $\mathfrak s$ reduces to the group $Spin(5) \times Spin(4)/(\Z/2)$.

Let $S=S_0 \oplus S_1, S_0=S_1=\q^2$ be a representation space of the spinor representation $\beta$ of  $Cl_{+5}$. We define the representation $\rho$ by  
 \[\rho([q,(u_0, u_1)])\phi = \beta(q) \cdot (\phi_0 u_0^{-1}, \phi_1 u_1^{-1})\]
for $\phi=(\phi_0,\phi_1) \in S \cong \q^2 \oplus \q^2$, $[q,(u_0, u_1)] \in Spin(5) \times (Sp(1) \times Sp(1))/(\Z/2) \cong Spin(5) \times Spin(4)/(\Z/2)$. 

We denote by $S'$ the associated vector bundle of the representation $\rho$.  
We give a $\Z/2\Z$ grading on $S' \oplus S'$ by declearing the subspace $\{(\varphi, \varphi) \in S'\oplus S'\}$ is the even part and the subspace $\{(\varphi, -\varphi) \in S'\oplus S'\}$ is the odd part. 
We give the following right $Cl_{+1}$ action: let $\epsilon_0$ be the generator of $Cl_{+1}$ and we define the right action of $\epsilon_0$ by $(\varphi_0,\varphi_1)\cdot \epsilon_0=(\varphi_0,-\varphi_1)$ for $(\varphi_0,\varphi_1) \in S' \oplus S'$. The left $Cl_{-1}$ action is given by $e_0 \cdot \phi := (\epsilon \phi)\cdot \epsilon_0$ where $\epsilon$ is the $\Z/2\Z$ grading operator. We define the Clifford multiplication $c'_n$ by the left Clifford multiplication of $Cl_{(5,4)}$ given by $\epsilon_j \cdot (\varphi_0,\varphi_1)=(\beta(\epsilon_j)\varphi_0,-\beta(\epsilon_j)\varphi_1)$ for $j=0, \dots 4$. Then $S' \oplus S'$ is $\Z/2\Z$ graded generalized spinor bundle with left $Cl_{-1}$ action of $\mathfrak s$.  
\end{definition}

\begin{remark} We can define the $c_{E_-}$ Clifford action on $S' \oplus S'$ such that that action is presereved by the isomorphism in the Lemma below. But we will not use $c_-$ in section \ref{moduli節}. 
\end{remark}

\begin{lemma}\label{reduction}
As a $\Z/2\Z$ graded generalized spinor bundle with left $Cl_{-1}$ action of $\mathfrak s$, $S'\oplus S' \cong \tilde S$. In particular, $\tilde{S}^+ \cong S'$. 
\end{lemma}
\begin{proof}

We identify 
$
Cl_{(5,4)} $ with $ (\q(2) \oplus \q(2))\otimes_{\R} \q(2)
$
 by using the isomorphism $f$ in Definition \ref{basis}. 
Let $\Phi$ be a map 
\[
(\q(2) \oplus \q(2)) \otimes_{\R} \q(2)
\ni (A_0, A_1) \otimes B
\mapsto (A_0 B^*, A_1 Ad(\alpha(\Gamma))(B^*)) \in \q(2) \oplus \q(2)
\]
 where $B^{\ast}$ is transpose matrix of $\bar B$. (We will denote by $\bar {\cdot}$ the quaternion conjugate.)
We define $Spin(5) \times Spin(4) /(\Z/2\Z)$ action on the space left hand side by 
\[
[q,(u_0,u_1)] (A_0,A_1) =(qA_0\begin{pmatrix}u_0^{-1}&0\\0&u_1^{-1}\end{pmatrix}, qA_1\begin{pmatrix}u_0^{-1}&0\\0&u_1^{-1}\end{pmatrix})
\] 
 for $[q,(u_0,u_1)] \in Spin(5) \times Spin(4) /(\Z/2\Z) \cong Spin(5) \times (Sp(1) \times Sp(1)) /(\Z/2\Z)$. For this action, we see $\Phi$ is a $Spin(5) \times Spin(4) /(\Z/2\Z)$ equivariant map because $\text{diag}(u_0^{-1},u_1^{-1})$ commutes with $\alpha(\Gamma)$.   
   
The map $\Phi$ is invariant under the right multiplication of $f(\epsilon_i e_i)$ for $i=1,2,3,4$. Therefore, the $-1$-eigenspaces of $f(\epsilon_i e_i)$ for some $i=1,2,3,4$ are the subspace of $\ker \Phi$. We see at once $\Phi$ is surjective and the dimension of $\q(2) \oplus \q(2)$ is equal to the dimension of $\tilde S$. Thus we have that the restriction of $\Phi$ to the representation space $\tilde S$ is a $Spin(5) \times Spin(4) /(\Z/2\Z)$ equivariant isomorphism.

We see the $Spin(5) \times Spin(4) /(\Z/2\Z)$ representation $\q(2) \oplus \q(2)$ is equivalent to the representation $S' \oplus S'$ through  
an isomorphism given by identifying each column of the matrix with an element of $S'_0, S'_1$. 
The $\Z/2\Z$ grading of $\tilde S$ is defined by the $\Z/2\Z$ grading of $Cl_{(5,4)}$. It is easy to see that this $\Z/2\Z$ grading of $\tilde S$ coincides with the $\Z/2\Z$ grading of $S' \oplus S'$. Moreover, we see at once $\Phi$ preserves the Clifford multiplication $c_{TZ}$ and $f(\epsilon_0)$. This completes the proof. 
\end{proof}

\begin{remark}\label{newremark}
If we change the basis of $S' \oplus S'$ by the isomorphism $(\phi_0, \phi_1) \mapsto \left(\frac{1}{2}(\phi_0+\phi_1), \frac{1}{2}(\phi_0-\phi_1)\right)$, we change the $\Z/2\Z$ grading of $S' \oplus S'$ and the Clifford multiplication $c_{TZ}$. The $\Z/2\Z$ grading is given by $S' \oplus 0$ is even and $0 \oplus S'$ is odd. The Clifford multiplication $c_{TZ}$ and the right $\epsilon_0$ action are given by the matrix
\[
c_{TZ}(\epsilon_j)= 
  \begin{pmatrix}
     0& \beta(\epsilon_j)\\
    \beta(\epsilon_j)&0 
  \end{pmatrix}, \;
\cdot \epsilon_0 =  
   \begin{pmatrix}
     0& 1\\
    1&0 
  \end{pmatrix}.
\]
Note that $c_{TZ}$ and right $\epsilon_0$ action commute.   
\end{remark}

\section{The orientability of $Pin^-(2)$ monopole moduli space}

\label{moduli節} 
\subsection{The orientability and mod $2$ indeces}

In section \ref{eg}, we consider the Spinor bundle and index of the $G^+(5,0,4)$ structure.
In this section, we translate the orientability of the $Pin^-(2)$ monopole moduli space into a $\bmod 2$ index on a five dimensional manifold
\footnote{Originally, Mikio Furuta translated the orientability of $Pin^-(2)$ monopole moduli space into a $\bmod 2$ index associated with a $3$ dimensional submanifold with $G^+(3,0,2)$ structure. Nobuhiro Nakamura wrote Furuta's idea in the note ~\cite{Nakamuranote}. The argument in this note is the neck streching argument. In this paper, we use Witten deformation to determmine the orientability of $Pin^-(2)$ monopole moduli space. }.
From Theorem \ref{main}, we give a topological criterion for the determinant line bundle on the ambient space of the $Pin^-(2)$ monopole moduli space to be trivial. 

\begin{definition}
We define a group $Spin^{c-}(4)$ by
\[
Spin^{c-}(4):=Spin(4) \times Pin^-(2)/ (\Z/2\Z).
\]
where $\Z/2\Z \cong \{ (1,1), (-1, -1)\} \in Spin(4) \times Pin^-(2)$.  
\end{definition}
\begin{remark}
By the definition above, we see that $Spin^{c-}(4)$ is a subgroup of $G^+(4,0,3)$ by using the natural embedding $Spin(4) \subset Cl_{(+4)}$.  
\end{remark}
\begin{definition}
We define a $Spin^{c-}$ structure on Riemannian $4$-manifold $X$ as a $G^+(4,0,3)$ structure whose structure group reduces to $Spin^{c-}$. 
\end{definition}
\begin{definition}
Let $X$ be a $4$-manifold and $\mathfrak{s}_X$ be a $Spin^{c-}$ structure on $X$. We will denote by $\tilde P$ a principal $Spin^{c-}(4)$ bundle given by the $Spin^{c-}$ structure $\mathfrak{s}_X$. 
\begin{itemize}
\item We identifies $\R^4$ with $\q$. We define the representation of $Spin^{c-}(4)$ 
\[
\Delta^{\pm} \colon Spin^{c-}(4) \to GL(4,\R)
\]
and 
\[
\rho \colon Spin^{c-}(4) \to GL(4,\R)
\]
by
\[
\Delta^{\pm}([q^+, q^-, u])\phi=q^{\pm}\phi u^{-1}, \; \rho([q^+, q^-, u])v=q^+ v (q^-)^{-1}.
\]
We will denote by $S^{\pm}$ the associated vector bundle given by the representation $\Delta^{\pm}$ and the principal $Spin^{c-}$ bundle $\tilde{P}$. The associated vector bundle given by the representaion $\rho$ is $TX$. 
\item We define a Clifford multiplication
\[
c \colon TX \otimes S^+ \to S^-
\]
 by $c(v \otimes \phi)=\bar{v}\phi$. We will denote by $c(v)\phi$ this Clifford action. We define
\[
c^* \colon TX \otimes S^- \to S^+ 
\]
by $c^*(v \otimes \phi)=v\phi$. By abuse of notation, we use the same letter $c(v)\phi$. We have $c(v)^2 =  \lvert v \rvert^2$ for $v \in T_x X$. 
\item A connection $A$ on $\tilde P$ is $Spin^{c-}$ connection if it satisfies 
\[
\nabla^{A}(c(X)\phi)=c(\nabla^{LC}X)\phi + c(X)\nabla^{A}\phi
\]
for $X \in \Gamma(TX), \phi \in \Gamma(S^+)$ where $\nabla^{LC}$ is the Levi-Civita connection of $X$. 

\item Let $A$ be a $Spin^{c-}$ connection. We define a Dirac operator $D^{\pm}_A$ associated to $A$ by the composition of following maps:
\[
\begin{CD}
\Gamma(S^{\pm}) @>{\nabla^{A}}>> \Gamma(TX^* \otimes S^{\pm}) \cong \Gamma(TX \otimes S^{\pm}) @>{c}>> \Gamma(S^{\mp})
\end{CD}.
\]
 
\end{itemize}
\end{definition}

The notion of the $Spin^{c-}$ structure is introduced by Nobuhiro Nakamura in ~\cite{Nakamura1}, ~\cite{Nakamura2} to define the $Pin^-(2)$ monopole equations and the $Pin^-(2)$ monopole invariant. 
The $Pin^-(2)$ monopole invariant is a $\Z/2\Z$-valued invariant.  
The matter when we  determine the orientability of the $Pin^-(2)$ monopole moduli space is whether a gauge transformation preserves an orientation of $\ind D_A$.  

To define a gauge transformation, we introduce some associated vector bundles and their Clifford actions. 

\begin{definition}
Let $X$ be a $4$-manifold and  $\mathfrak{s}_X$ be a $Spin^{c-}$ structure on $X$. Let $\tilde P$ be the principal $Spin^{c-}(4)$ bundle given by the $Spin^{c-}$ structure $\mathfrak{s}_X$. 
\begin{itemize}
\item We identify $\R^2$ with the subspace $\{z \in \q \mid z=x+iy,\; x, y \in \R \}\cong \C$. We define the real $Spin^{c-}(4)$ representation
\[
\rho'_0 \colon Spin^{c-}(4) \to GL(2,\R)
\]
 by $\rho'_0([q^+, q^-, u])z=uzu^{-1}$ where the multiplication is that of $\q$. 
\item We identify $\R^2$ with the subspace $\{ jw \in \q \mid w \in \C \}$. We define the real $Spin^{c-}(4)$ representation
\[
\rho'_1 \colon Spin^{c-}(4) \to GL(2,\R)
\]
 by $\rho'_1([q^+, q^-, u])jw=ujwu^{-1}$ where the multiplication is that of $\q$. 
\end{itemize}
It is easy to check the well-definedness of the above definition. We will denote by $\tilde{\C}, E$ the associated vector bundles given by $\tilde P$ and the representations $\rho'_0, \rho'_1$ respectively. 
\end{definition}

\begin{definition}
We define the multiplication $S \otimes (\tilde{\C}\oplus E) \to S$ as follows. Note that the representation space of $\rho_0 \oplus \rho_1$ and $\Delta^+ \oplus \Delta^-$ is $\q$ and $\q^2=\q_+ \oplus \q_-$ respectively. We define a multiplication of elements in these representation space by 
\[
 \phi \cdot v :=\phi v
\] for $v \in \q$ and $\phi \in \q^2$. This is a $Spin^{c-}(4)$ equivariant multiplication. This defines a multiplication $S=S^+ \oplus S^-$ and $\tilde{\C}\oplus E$. 
We define the multiplication $(\tilde{\C} \oplus E) \otimes (\tilde{\C} \oplus E) \to \tilde{\C} \oplus E$ in the same way.
\end{definition}

\begin{remark}
Regarding the $Spin^{c-}$ structure $\mathfrak s_X$ as a  $G^+(4,0,3)$ structure, the vector bundle $\det E \oplus E$ is the vector bundle $E_-$ associated with $\mathfrak s_X$. (c.f Definition\ref{def-G-+}. ) From the definition of $\tilde{\C}$ and $E$, we have $\text{Im}(\tilde{\C}) \cong \det E$. From now on, we fix an isomorphism $\text{Im}(\tilde{\C}) \cong \det E$ and identify $\text{Im}(\tilde{\C})$ with $\det E$. If we take another isomorphism, arguments do not change. 
\end{remark}

\begin{definition}
We call an element of $\{ u \in \Gamma(\tilde{\C}) \mid \lvert u \rvert = 1 \}$ a gauge transformation. 
\end{definition}

Now we state the topological method of evaluating the orientability of $Pin^-(2)$ monopole space. 

\begin{theorem}\label{mukimuki}
Let $u$ be a gauge transformation. We perturb $u$ if necessary and we may assume that $-1$ is a regular value of $u$. Let  $h$ be a section of $E$ such that $h$ is transverse to the zero section and its submanifold $u^{-1}(-1) \subset X$. Then there is a natural $Spin$ structure on $h^{-1}(0) \cap u^{-1}(-1)$. This defines an element of $1$ dimensional $Spin$ bordism group $\Omega_1^{Spin}(pt) \cong \Z/2\Z$ and we denote by $\text{t-}\ind(\mathfrak s_X, u)$ this element. The following statements are equivaleant: the gauge transformation $u$ preserves the orientation the index bundle $\{\ind(D_A)\}_A$ on the configuation space and $\text{t-}\ind(\mathfrak s_X, u)$ is trivial. 
\end{theorem}

\begin{corollary}
If we have $\text{t-}\ind(\mathfrak s_X, u)=0$ for all gauge transformation $u$, the $Pin^-(2)$ monopole moduli space is orientable. 
\end{corollary}
We begin the preparation of the proof of Theorem \ref{mukimuki} 

\begin{definition}
Let $X$ be a $4$-manifold and  $\mathfrak{s}_X$ be a $Spin^{c-}$ structure on $X$. Let $S=S^+ \oplus S^-$ be the spinor bundle of $\mathfrak s_X$ and $u$ be a gauge transformation. 
 
\begin{itemize}
\item We define the vector bundle $L$ on $X \times S^1$ as follows. Let $\pi \colon [0,1] \times X \to X$ be the projection. We introduce an equivalence relation on $\pi^* \tilde{\C}$ by $(0,z) \sim (1,zu)$ for $z \in \tilde{\C}$ and $L$ is a quotient of this equivalence relation $L=\pi^* \tilde{\C}/\sim$. Note that $L$ has the canonical left $\pi^* \tilde{\C}$ action.

	\item Let $\tilde{\pi} \colon  S^1 \times X \to X$ be the projection. By abuse of notation, we use the same letter $\tilde{\C}$ for $\tilde{\pi}^* \tilde{\C}$. 
\item $V=V^+ \oplus V^-$ is the $\Z/2\Z$ graded vector bundle given by $V=\tilde{\pi}^* S \otimes_{\tilde{\C}} (L \oplus \tilde{\C})$, $V^+=\tilde{\pi}^* S$ and $V^-=\tilde{\pi}^* S \otimes_{\tilde{\C}}L$ where $\otimes_{\C}$ is a tensor product of $\tilde{\C}$ module. 
\end{itemize}
\end{definition}

\begin{lemma}\label{gaugeindex}
We define a skew-adjoint Dirac type operator $D$ on $V$ as follows: 
\begin{itemize}
\item Let $A$ be a $Spin^{c-}$ connection on $X$ and $D_A$ is the Dirac operator on $X$ given by $A$.
\item Let $A_t$ be a one-parameter family of $Spin^{c-}$ connection on $X$ such that $A_t = A$ for $t<1/3$ and $A_t = u^* A$ for $t>2/3$.  
\item Let $\epsilon$ be the $\Z/2\Z$ grading operator of $S$. Let $\epsilon'$ be the operator on $L \oplus \tilde{\C}$ given by $1 \oplus (-1)$. 
\item We define the Dirac type operator on $X \times S^1$ by
\[
D=D_t \otimes pr_{L}+ D_A \otimes pr_{\tilde{\C}} + \epsilon \partial_t \otimes \epsilon'
\]
where $t$ is the coordinate of $S^1$ and $pr_{L}, pr_{\tilde{\C}}$ are the projections to $L, \tilde{\C}$ respectively. 
\end{itemize}
Then the operator $D$ is well defined on $ S^1 \times X$. 
It is equivalent that $u$ preserves an orientation of the index bundle $\{\ind(D_A)\}$ on the configuation space and $\dim \ker D \mod 2 = 0$. 
\end{lemma}
\begin{proof}
The four tuples $(D_t \otimes pr_{L}+ D_A \otimes pr_{\tilde{\C}},  \epsilon \otimes \epsilon', 0, L^2(X \times S^1, V))$ defines a family index in $KO^{0}(S^1, pt)$. This family index coincides with $\ind(D_t) - \ind(D_A)$ when we use the definition of $KO^0(S^1, pt)$ as the subgroup of the Grothendiek group of real vector bundles on $S^1$. 
We see at once the family index $\ind(D_t) - \ind(D_A)$ is trivial if and only if $u$ preserves an orientation of $\ind(D_A)$ by the definition of $V$ and $D_t$. 


We see that this family index coincides with $\dim \ker D \mod 2 \in \Z/2\Z \cong KO^0(S^1, pt)$ from index theory. 
\end{proof}

\begin{definition}
We define the $G^+(5,0,4)$ structure $\mathfrak{s}_{Z}$ on $Z=X \times S^1$ as follows: 
Let $\tilde{P}$ be the principal $Spin^{c-}(4)$ bundle on $X$ given by $Spin^{c-}$ structure $\mathfrak s_X$ and $\pi \colon [0,1] \times X \to X$ be the projection. 
We will denote by $\iota$ an embedding $Spin^{c-}(4) \cong Spin(4) \times_{\Z/2\Z} Pin^-(2) \to Spin(5) \times_{\Z/2\Z} Spin(4) \cong Spin(5) \times_{\Z/2\Z} (Sp(1) \times Sp(1))$ given by $[q, u'] \mapsto [i(q), (u',u')]$ where $i$ is a embedding $Spin(4) \to Spin(5)$ which is a lift of a map $A \mapsto \text{diag}(1,A)$. Let $u$ be a gauge transformation of $\mathfrak s_X$.
We define a principal $Spin(5) \times_{\Z/2\Z} Spin(4)$ bundle $\tilde{P}_Z$ on $Z$ by $\pi^* \tilde{P} \times_{\iota} (Spin(5) \times_{\Z/2\Z} Spin(4))/ \sim$ where $\sim$ is an equivalent relation given by 
\[
(0, p) \sim (1, pu),\; p \in \tilde{P} \times_{\iota} (Spin(5) \times_{\Z/2\Z} Spin(4)).
\] 
It is easy to see that $\tilde{P}_Z$ defines a $Spin$ structure of $TZ \oplus L  \oplus \pi^* E$. We will denote by $\mathfrak{s}_{Z}$ a $G^+(5,0,4)$ structure on $Z$ given by $\tilde{P}_Z$. 
\end{definition}

\begin{lemma}\label{indexindex}
The mod $2$ index $\dim \ker D \mod 2$ of the skew-adjoint operator $D$ coincides with $\ind (\mathfrak{s}_{Z}) \in KO^{-1}(pt) \cong \Z/2\Z$ where $\ind(\mathfrak s_Z)$ is the index of $\mathfrak{s}_{Z}$.  
\end{lemma}
\begin{proof}
By the definition of $V$, we see that $V$ is a Spinor bundle $S'$ of $G^+(5,0,4)$ structure $\mathfrak{s}_Z$ given in section \ref{G^+(5,0,4)}. This lemma follows from Lemma \ref{reduction} and Remark \ref{newremark}. 
\end{proof}

We have the following proposition from Theorem \ref{main}. 

\begin{proposition}
Let $X$ be a closed Riemannian $4$-manifold and $\mathfrak s$ be a $Spin^{c-}$ structure on $X$. We will denote by $u \in \Gamma(\tilde \C)$ a gauge transformation. The following statements are equivalent : $u$ preserves the orientation the index bundle $\{\ind(D_A)\}$ and a $Spin$ structure induced on the zero locus of a transverse section of the vector bundle $L \oplus \pi^* E$ on $X\times S^1$ from Theorem \ref{loc} is trivial. 
\end{proposition}

We consider the $Spin$ structure on the zero locus of the transverse section of $L \oplus \pi^* E$. 

\begin{lemma}\label{spin1}
If it is necessary, we perturb $u$ by homotopy and we may assume that $-1$ is a regular value of $u$.  Let  $h$ be a section of $E$ such that $h$ is transverse to the zero section and its submanifold $u^{-1}(-1) \subset X$. Let $C=h^{-1}(0) \cap u^{-1}(-1)$ and $U(C)$ is its tubular neighborhood. From Theorem \ref{loc}, we have a spin structure $\mathfrak s_C$ on $C \subset U(C)$ introduced by the section $\text{Im}(u) \oplus h \in \Gamma((\det E \oplus E)|_{U(C)})$. Then there exists a transverse section of the vector bundle $L \oplus \pi^* E$ whose zero locus is $(h^{-1}(0) \cap u^{-1}(-1)) \times \{1/2\} \subset X \times \{1/2\}$ and the $Spin$ structure on $C \times \{1/2\}$ given in  Theorem \ref{loc} coincides with $\mathfrak s_C$. 
\end{lemma}
\begin{proof}
The transversality of $\text{Im}(u) \oplus h \in \Gamma((\det E \oplus E)|_{U(C)})$ follows from the assumption of $u$ and $h$. Then $\mathfrak s_C$ is well defined. 

To define a section of $L \times \pi^* E$, it is sufficient to take a section $s=(s_0, s_1) \in \Gamma((\tilde{\C} \oplus E)\times [0,1])$ such that $s_0(0) \cdot u=s_0(1)$. We define  a section $s=(s_0, s_1)$ as follows : 
\begin{align*}
s_0(t) &=(1-t)+tu,  \\
s_1(t) &=h.
\end{align*} 
From the definition of $s$, we see that $s^{-1}(0)=(h^{-1}(0) \cap u^{-1}(-1)) \times \{1/2\} \subset X \times \{1/2\}$ and $s$ is transverse to the zero section. The normal bundle of $s^{-1}(0)$ splits into the $[0,1]$ direction and the $X$ direction. The $[0,1]$ direction is trivial rank one vector bundle, and the $X$ direction is isomorphic to the vector bundle $\det E \oplus E$. On $s^{-1}(0)$, the real part of $s_0$ is a section of a summand of the normal bundle which is the $[0,1]$ direction, and $\text{Im}(s_0) \oplus s_1= \text{Im}(u)/2 \oplus h$ is a section of $\det E \oplus E$. Then we have a $Spin$ structure on $s^{-1}(0)$ given by $s$ and $\mathfrak s_Z$ coincides with $\mathfrak s_C$.  
\end{proof}

Now we prove Theorem \ref{mukimuki}.

\begin{proof}[Proof of Theorem \ref{mukimuki}]
From Lemma \ref{spin1}, a gauge transformation $u$ preserves the orientation the index bundle $\{\ind(D_A)\}_A$ on the configuation space  
if and only if the index of the $Spin$ structure $\mathfrak s_C$ on $h^{-1}(0) \cap u^{-1}(-1)$ is trivial. The index gives the isomorphism $\Omega_1^{Spin}(pt) \cong \Z/2\Z$. Thus we prove Theorem \ref{mukimuki}.  
\end{proof}

\begin{remark}\label{Pinc+}
A slight change in the proof of Lemma \ref{spin1} shows that there is a $G^+(3,0,2)$ structure  ($Pin^{\tilde c}_+$ structure) on $Y := u^{-1}(-1) \subset X$ induced by $\mathfrak s_X$. The vector bundle $E |_Y$ coincides with $E_-$ given by this $G^+(3,0,2)$ structure. Thus we have that the index of the $Spin$ structure on $h^{-1}(0) \cap Y$ coincides with the index of this $G^+(3,0,2)$ structure on $Y$ from Theorem \ref{main}. From Theorem \ref{mukimuki}, $u$ preserves an orientation of moduli space if and only if the index of this $G^+(3,0,2)$ structure on $Y = u^{-1}(-1)$ is trivial. This is the statement proved by Mikio Furuta. 
\end{remark}

\subsection{Examples}
In this section, we give an example of four manifold with $Spin^{c-}$ such that there exists a gauge transformation which reverses an orientation of $\{ \ind D_A \}$. 
\begin{proposition}\label{construction}
Let $Y$ be a $3$ dimensional closed Riemannian manifold and $\mathfrak s$ be a $G^+(3,0,2)$ structure ($Pin^{\tilde c}_+$ structure ) on $Y$ such that the index of $\mathfrak s$ is non-trivial. We denote by $X$ a four manifold given by gluing two copies of the disk bundle of $\det TY$ along boundaries. There is a $Spin^{c-}$ structure on $X$ and a gauge transformation which reverse the orientation of $\{ \ind D_A \}$. 
\end{proposition}
\begin{proof}
First, we construct a $Spin^{c-}$ structure on $X$. 
\begin{itemize}
\item Let $l =\det TY$ and $\pi \colon l \to Y$ be the projection. We will denote by $E$ a vector bundle $E_-$ associated with the $G^+(3,0,2)$ structure $\mathfrak s$. Note that $\det E \cong l$ and $\mathfrak s$ is given by a $Spin$ structure of $\det TY \oplus TY \oplus \det E \oplus E$.  
\item We will denote by $\mathfrak{s}'$ a $Spin^{c-}$ structure on the total space $l$ given by the $Spin$ structure of $\pi^* l \oplus TY \oplus \pi^* l \oplus \pi^* E$ induced by $\mathfrak s$. 
\item Let $D(l)$ be the disk bundle of $l$ and $S(l)$ be the sphere bundle of $l$. We choose a canonical trivialization of $\pi^* l$ on $S(l)$. Hence we have that $E$ and $TS(l)$ are orientable on $S(l)$. Thus the restriction of $\mathfrak{s}'$ to $S(l)$ induces a $Spin^c$ structure. 
\item $S(l)$ is a double cover of $Y$ and the covering transformation $\tau$ reverses the orientation of $S(l)$. We glue two copies of $D(l)$ along $S(l)$ by the map $\tau$. Let $X=D(l) \cup_{\tau} D(l)$. $X$ is an oriented closed manifold. 
\item We glue $Spin^c$ structure on each $S(l)$ to give a $Spin^{c-}$ structure $\mathfrak s_X$ on $X$ such that the restriction of $\mathfrak s_X$ to each $D(l)$ coincides with $\mathfrak{s}'$. 
\end{itemize}
Second, we give a gauge transformation $u$ which reverses the orientation of $\ind D_A$. 
\begin{itemize}
\item We will denote by $f$ the tautological section of $\pi^* l$ on $l$. 
\item On the open subset $l\setminus Y$, $\pi^* l$ have the canonical trivialization. In this trivialization, we have $f(v)=\lvert v \rvert$ for $v \in l \setminus Y$. 
\item Deform $f$ in the area $\lvert v \rvert \ge 1/2$ and we assume that $f(v)=1$ for $\lvert v \rvert \ge 2/3$. 
\item We define $s(v)=-\exp(i\pi f(v))$ as follows: There is the natural isomorphism $l \otimes l \cong \R$ and using this isomorphism we have $f^{2n} \in \R$, $f^{2n+1} \in l$. We define $\exp(i\pi f(v)) \in S(\R \otimes \sqrt{-1}l)$ using Taylor expansion of exponential function. 
\item Note that $s=1$ around $S(l)$ and we extend on $X$ by $1$ on another $D(l)$. We have $s^{-1}(-1)=Y$. 
\item The index of $\mathfrak s$ on $Y$ is nontrivial and from Theorem \ref{mukimuki} and Remark \ref{Pinc+}, $u$ reverses the orientation of $\ind D_A$. 
\end{itemize} \end{proof}

We give the explicit example of $Y$ in Proposition \ref{construction}.  
\begin{lemma}\label{s_0}
There is a $G^+(3,0,2)$ structure $\mathfrak{s}_0$ on $\R P^2 \times S^1$ whose index is non-trivial. 
\end{lemma}
\begin{proof}
For simplicity of notations, we omit the notation of the pull-back of projections. From Lemma \ref{spin0}, we give a $Spin$ structure on the vector bundle
\[
\det T\R P^2 \oplus TS^1 \oplus T\R P^2 \oplus \det T\R P^2 \oplus T\R P^2
\] and we give the $G^+(3,0,2)$ structure $\mathfrak{s}_0$ by Lemma \ref{emb}. Note that $E_-= T\R P^2$. We take a transverse section of $T\R P^2$ on $\R P^2$ such that whose zero locus is a single point on $\R P^2$.  By pulling back this section, we have a transverse section $h$ of $E_-$ whose zero locus is $S^1 \times \{pt\}$. We immediately see that the $Spin$ structure on $h^{-1}(0) \cong S^1$ induced by $h$ from Theorem \ref{loc} is given by the product $S^1 \times Spin(1)$. This is the nontrivial element in the $1$-dimensional $Spin$ bordism group. From Theorem \ref{main}, we have that the index of $\mathfrak{s}_0$ is nontrivial. 
\end{proof}

From Proposition \ref{construction} and Lemma \ref{s_0}, we deduce following corollary. 
\begin{corollary}\label{exexample}
We set $X=(\R P^3 \sharp \R P^3) \times S^1$ and $\mathfrak s=\mathfrak s_0$. The deternminant bundle on the ambient space of the moduli space is not orientable. 
\end{corollary}

This manifold is diffeomorphic to $P\gamma \times S^1$ where $\gamma$ is the tautological bundle of $\R P^2$ and $P\gamma$ is its projectivization.  
\begin{remark}
The gluing construction in the proof of Lemma \ref{construction} can be generalized in the case of gluing two $4$ dimensional $Spin^{c-}$ manifolds with boundary. If the restrictions of the $Spin^{c-}$ structures to their boundaries induce $Spin^c$ structures and their boundaries are diffeomorphic by a map which preserves the orientation and the $Spin^c$ structure, our construction works.  
\end{remark}
\begin{remark}
In the case $Y=\R P^2 \times S^1$ and $\mathfrak s =\mathfrak{s}_0$, we have $S(l)=S^2 \times S^1$. We cannot glue $D(l)$ by $D^3 \times S^1$ along $S(l)$ because the first Chern class of $Spin^c$ structure on $S(l)$ is the Euler class of $TS^2$ and this cannot be extended to $D^3 \times S^1$.  
\end{remark}

\section{Appendix}
In this section, we prove Proposition \ref{Wittendeformation}. We follow the notations of  subsection \ref{Witten_deformation}
\begin{lemma}\label{ineq}
We assume that $\phi \in \Gamma(Y, \Slash{\mathfrak S})$ satisfies $\lVert \Slash{D}_m \phi \rVert \le \lambda \lVert \phi \rVert $. 
There are functions $A_h(m,\lambda), B_h(m, \lambda)$ of positive real numbers $m, \lambda$ depending on $h$ such that if we fix a value $\lambda$, they satisfy 
\[
A_h(m, \lambda) \to 0, \; B_h(m, \lambda) \to 1
\] when $m \to \infty$. The functions $A_h(m,\lambda), B_h(m, \lambda)$ satisfy the following inequalities where $V= \{ z \in U(C) \cong B(N(C)) \mid \lvert z \rvert > 1/2 \} \cup U(C)^c$:
\begin{align*}
\int_V \lvert \phi \rvert^2 \le  A_h(m, \lambda) \int_Y \lvert \phi \rvert^2, \\
B_h(m,\lambda) \int_Y \lvert \phi \rvert^2 \le \int_Y \lvert \rho \phi \rvert^2.
\end{align*}
\end{lemma}
\begin{proof}
We have the following estimate: 
\begin{align*}
\lambda^2 \int_Y \lvert \phi \rvert^2 &\ge \int_Y \lvert \Slash{D}_m \phi \rvert^2 \\
&=\int_Y \langle -\Slash{D}_m^2 \phi,\phi \rangle \\
&=\int_Y \langle (-(\Slash D)^2 -m\{\Slash D,h\} -m^2h^2)\phi,\phi \rangle \\
&=\int_Y \lvert \Slash D \phi \rvert^2  +\int_Y m^2\lvert h \phi \rvert^2  
-\int_Y m\langle \{\Slash D,h\},\phi \rangle \\
&\ge \int_V m^2\lvert h \phi \rvert^2 
- \int_Y mC_0\lVert dh \rVert_{\infty} \lvert \phi \rvert^2 \\
&\ge \frac{m^2}{2}\int_V \lvert \phi \rvert^2 -mC_0\lVert dh \rVert_{\infty} \int_Y \lvert \phi \rvert^2.
\end{align*}
We define $A_h(m,\lambda)=2\frac{mC_0\lVert dh \rVert_{\infty}+\lambda^2}{m^2}$ and we have the first inequality where $C_0$ is a constant only depending on the principal symbol of $D$ (Clifford action). We have the second inequality by the following estimate:
\begin{align*}
\left(\int_Y \lvert \rho \phi \rvert^2 \right)^{1/2} 
& \ge \left( \int_Y \lvert \phi \rvert^2 \right)^{1/2} -\left( \int_Y \lvert (1-\rho) \phi \rvert^2 \right)^{1/2} \\
& \ge \left( \int_Y \lvert \phi \rvert^2 \right)^{1/2} -\left( \int_V \lvert  \phi \rvert^2 \right)^{1/2} \\
& \ge \left( \int_Y \lvert \phi \rvert^2 \right)^{1/2} - \left(A_h(m, \lambda) \int_Y \lvert \phi \rvert^2 \right)^{1/2} \\
&=\left( 1-\sqrt{A_h(m, \lambda)}  \right) \left( \int_Y \lvert \phi \rvert^2 \right)^{1/2}.
\end{align*}
Setting $B_h(m, \lambda)=(1-\sqrt{A_h(m, \lambda)})^2$, the proof is completed. 
\end{proof}
A slight change of the proof of the above Lemma actually shows the following Lemma. 
\begin{lemma}\label{ineq2}
We will denote by $\mathcal{S} (N(C), \Slash{\mathfrak S}) $ a set of rapidly decreasing sections of $\Slash{\mathfrak S}$ on the total space of the vector bundle $N(C)$. 
We assume a section $\psi \in \mathcal{S} (N(C), \Slash{\mathfrak S})$ satisfies $\lVert \Slash{D}_m \phi \rVert \le \lambda \lVert \phi \rVert $. 
There are functions $A'(m,\lambda), B'(m, \lambda)$ of positive real numbers $m, \lambda$ such that if we fix a value $\lambda$, they satisfy 
\[
A'(m, \lambda) \to 0, \; B'(m, \lambda) \to 1
\] when $m \to \infty$. The functions $A'(m,\lambda), B'(m, \lambda)$ satisfy the following inequalities where $V'= \{ z \in N(C) \mid \lvert z \rvert > 1/2 \}$:
\begin{itemize}
\item \[
\int_{V'} \lvert \phi \rvert^2 \le  A'(m, \lambda) \int_{N(C)} \lvert \phi \rvert^2
\]

\item \[
B'(m,\lambda) \int_{N(C)} \lvert \phi \rvert^2 \le \int_{N(C)} \lvert \rho \phi \rvert^2.
\]
\end{itemize}
\end{lemma}
Note that the perturbation term $h$ of $\Slash{D}_m$ on $N(C)$ satisfies $\lvert \{ \Slash{D}, h \} \rvert=2$.

\begin{lemma}\label{<}
We assume that $\lambda$ is smaller than a constant given by the principal symbol of $D$ and the differentation of $\rho$. We assume that $m$ is large enough. 
Let $\Pi'$ be the orthogonal projection from $L^2(N(C), \Slash{\mathfrak S})$ to $H$. 
Then the map 
\[
\begin{array}{r@{\,\,}c@{\,\,}c@{\,\,}c}
&\mathcal{H}_m^{\lambda} &\longrightarrow & H \\
&\rotatebox{90}{$\in$}&&\rotatebox{90}{$\in$}\\
& \phi &\longmapsto & \Pi'(\rho \phi)
\end{array}
\]
is injective.
\end{lemma}
\begin{proof}
If a $\phi \in \mathcal{H}_m^{\lambda}$ with its $L^2$ norm is $1$ satisfies $\Pi'(\rho \phi)=0$, we have 
\[
\lVert \Slash{D}_m \rho \phi \rVert_{L^2(N(C), \Slash{\mathfrak S})} \ge \lambda_C \lVert \rho \phi \rVert_{L^2(N(C), \Slash{\mathfrak S})}
\]
from Lemma \ref{eigenvalue} and $m$ is large enough. The support of the function $\rho$ is containd in $U=\{ z \in N(C) \mid \lvert z \rvert<1\}$ and we identify $U$ with $U(C) \subset Y$. We regard $\rho \phi$ as a section on $Y$ and we have 
\[
\lVert \Slash{D}_m \rho \phi \rVert_{L^2(Y, \Slash{\mathfrak S})} \ge \lambda_C \lVert \rho \phi \rVert_{L^2(Y, \Slash{\mathfrak S})}.
\]
But we have the following estimate from Lemma \ref{ineq}:   
\begin{align*}
\int_Y \lvert D_m \rho \phi \rvert^2 &\le \int_Y \lvert [\Slash D, \rho] \phi \rvert^2 
+ \int_{Y} \lvert \rho D_m \phi \rvert^2 \\
&\le C_0\lVert d \rho \rVert_{\infty}^2 \int_{U\cap V} \lvert \phi \rvert^2 
+\int_{Y} \lvert D_m \phi \rvert^2  \\
&\le C_0\lVert d \rho \rVert_{\infty}^2 A_h(m,\lambda) \int_Y \lvert \phi \rvert^2 
+ \lambda^2 \int_Y \lvert \phi \rvert^2 \\
&\le C_0\lVert d \rho \rVert_{\infty}^2 \left( A_h(m, \lambda)+ \lambda^2 \right) B_h(m, \lambda)^{-1} \int_Y \lvert \rho \phi \rvert^2.
\end{align*}
Provided $m$ is large enough, the coefficient of $\int_Y \lvert \rho \phi \rvert^2$ tends to $C_0\lVert d \rho \rVert_{\infty}^2 \lambda^2$. If this constant is smaller than $\lambda_C$, the above estimate contradicts the inequality $\lVert D_m \rho \phi \rVert \ge \lambda_C \lVert \rho \phi \rVert$. 
\end{proof}
We fix the value $\lambda$ so that $2 C\lVert d \rho \rVert_{\infty}^2 \lambda^2 < \lambda_C$.  
\begin{lemma}\label{>}
We assume that $m$ is large enough. 
Let $\Pi_m^{\lambda}$ be the orthogonal projection from $L^2(Y,\Slash{\mathfrak S})$ to $\mathcal{H}_m^{\lambda}$. The map 
\[
\begin{array}{r@{\,\,}c@{\,\,}c@{\,\,}c}
&H &\longrightarrow &\mathcal{H}_m^{\lambda}\\
&\rotatebox{90}{$\in$}&&\rotatebox{90}{$\in$}\\
& \psi &\longmapsto & \Pi_m^{\lambda}(\rho \psi)
\end{array}
\]
is injective. 
\end{lemma}
\begin{proof}
If the map above is not injective, we have
\[
\lVert D_m \rho \psi \rVert_{L^2(Y, \Slash{\mathfrak S})} \ge \lambda \lVert \rho \psi \rVert_{L^2(Y, \Slash{\mathfrak S})}
\]
for some $\psi \in H$. 
On the other hand, elements in $H$ are rapidly decreasing sections hence we use Lemma \ref{ineq2}. Thus we have the following estimate by a similar argument of the proof of the Lemma
\ref{<}: 
\[
\int_{N(C)} \lvert D_m \rho \psi \rvert^2 \le C\lVert d \rho \rVert_{\infty}^2 A'(m, 0)B'(m, 0)^{-1} \int_{N(C)} \lvert \rho \psi \rvert^2.
\] 
If $m$ is large enough, the above estimate contradicts the inequality $\lVert D_m \rho \phi \rVert \ge \lambda \lVert \rho \phi \rVert$. 
\end{proof}

Now we prove Proposition \ref{Wittendeformation}.

\begin{proof}[proof of \ref{Wittendeformation}]
We have $\dim H \ge \dim \mathcal{H}_m^{\lambda}$ from Lemma \ref{<} and  we have $\dim H \le \dim \mathcal{H}_m^{\lambda}$  from Lemma \ref{>}. Thus we have the maps of Lemma \ref{<} and Lemma \ref{>} are isomorphisms. In particular, the map in Proposition \ref{Wittendeformation} is the same as the map in Lemma \ref{>} and it is an isomorphism. Moreover, it is easy to see that this map preserves the $\Z/2$ gradings and the left $Cl_{(s^-, n+s^+)}$ actions. This completes the proof.
\end{proof}


\begin{thebibliography}{99}
\bibitem{Atiyah1}
M.F.Atiyah,
{Vector fields on Manifolds},
{Springer},
(1969).

\bibitem{Atiyah2}
I.M.Atiyah,
{\em Riemann surfaces and spin structures},
{Annales scientifiques de l'{\'E}cole Normale Sup{\'e}rieure},
  {\bf 4},
  no.{1},
  :{47--62},
  (1971).
  
\bibitem{ABSCliff}
M.F.Atiyah, R.Bott, A.Shapiro.
{\em Clifford modules},
 Topology,
  {3},
  {3--38},
  (1964),
  {Pergamon}.

\bibitem{AtiyahSinger1}
M.F.Atiyah, I.M. Singer.
{\em Index theory for skew-adjoint Fredholm operators},
  Publications Math{\'e}matiques de l'Institut des Hautes {\'E}tudes Scientifiques,
  {\bf 37},
  no.1:
  {5--26},
  (1969),
  Springer.

\bibitem{AtiyahSinger2}
M.~F. Atiyah and I.~M. Singer.
{\em The index of elliptic operators. {V}.}
Ann. of Math. (1), {\bf 93}:139--149, (1971).


\bibitem{Donaldson-ori}
{S.Donaldson,},
{\em The orientation of Yang-Mills moduli spaces and 4-manifold topology},
  {Journal of Differential Geometry},
  {\bf 26},
  {no. 3},
  :{397--428},
  {1987},
 {Lehigh University}.

\bibitem{Donaldson2}
{S.Donaldson, P.Kronheimer},
{The geometry of four-manifolds},
 ({1997}),
  {Oxford university press}.

 
  
\bibitem{FO04}
J.L.Fast, S.Ochanine.
{\em On the {$K{\rm O}$} characteristic cycle of a {${\rm Spin}^c$} manifold.}
Manuscripta Math., {\bf 115}(1):73--83, (2004).


\bibitem{FreedHopkins1}
{D.Freed},{M.Hopkins}, 
{\em Reflection positivity and invertible topological phases},
Geometry \& Topology,. {\bf 25}(3):1165-1330, (2021).

\bibitem{freedhopkins2}
{D.Freed},{M.Hopkins}, 
  {\em Invertible phases of matter with spatial symmetry},
  {arXiv preprint arXiv:1901.06419},
  (2019).
  

\bibitem{FreedMoore}
{D.S.Freed},{G.W.Moore},
{\em Twisted equivariant matter},
 {Annales Henri Poincar{\'e}},
  {\bf 14},
  no.{8},
  :{1927--2023},
  (2013),
  {Springer}.



\bibitem{FFMOYY}
H.Fukaya, M.Furuta, S.Matsuo, T.Onogi, S.Yamaguti, M.Yamasita, 
{\em The Atiyah--Patodi--Singer Index and Domain-Wall Fermion Dirac Operators}, 
  {Communications in Mathematical Physics},
  {\bf 380},
  no.{3},
  :{1295--1311},
  (2020).
  {Springer}.

\bibitem{Furuta}
{M.Furuta,},
 {Index theorem. 1},
  {\bf 1},
  {2007},
  {American Mathematical Soc.}

\bibitem{FK00}
M.~Furuta and Y.~Kametani.
{\em Equivariant version of rochlin-type congruences.}
J. Math. Soc. Japan.
 {\bf 66},
  no.{1},
  :{205--221},
  {(2014)}.


\bibitem{Gomi}
K.Gomi, 
{\em Freed-Moore K-theory},
{arXiv preprint arXiv:1705.09134},
(2017). 

  
\bibitem{Hayashi}
{S.~Hayashi},
{\em Localization of Dirac Operators on 4n+ 2 Dimensional Open Spin\^{c}  Manifolds},
  {arXiv preprint arXiv:1306.0389},
  (2013).
  
\bibitem{Joyce-Tanaka-Upmire}
{D.~Joyce, Y.~Tanaka, M.~Upmeier},
{\em On orientations for gauge-theoretic moduli spaces},
  {Advances in Mathematics},
  {\bf 362},
  :{106957},
 {2020},
 {Elsevier}.
  

\bibitem{KirbyTaylor}
{R.Kirby, L.Taylor},
 {Pin structures on low-dimensional manifolds. Geometry of lowdimensional manifolds},
  {Geometry of low-dimensional manifolds, 2},
  {177--242},
  (1990).
\bibitem{kronheimer2005four}
  {P.B.Kronheimer},
  {\em Four-manifold invariants from higher-rank bundles},
{Journal of Differential Geometry},
{\bf 70},
no.{1},
:{59--112},
{2005},
{Lehigh University}.

\bibitem{kronheimer2011knot},
  {P.B.Kronheimer, T.S.Mrowka},
  {\em Knot homology groups from instantons},
  {Journal of Topology},
  {\bf 4},
  no.{4},
  :{835--918},
  {2011},
  {London Mathematical Society}.
 

\bibitem{Kubota}
{Y.Kubota},
{\em Controlled Topological Phases and Bulk-edge Correspondence},
{Communications in Mathematical Physics},
  {\bf 349},
  no.{2},
  :{493--525},
  (2017),
  {Springer}.

\bibitem{LM89}
H.
Lawson, Jr.,
M.
Michelsohn.
Spin geometry, volume~38 of Princeton Mathematical Series.
Princeton University Press, Princeton, NJ, (1989).


\bibitem{Nakamura1}
N.Nakamura.
{\em {$Pin^-(2)$}monopole equations and intersection forms with local coefficients of 4-manifolds}.
Math. Ann. {\bf 357}, : 915--939(2013).
\bibitem{Nakamura2}
N.Nakamura,
{\em {$Pin^-(2)$}monopole invariants},
J.Differential Geom. {\bf 101},no.3, :507--549.(2015).

\bibitem{Nakamuranote}
N.Nakamura,
{\em A note on the orientability of the moduli space of the $Pin^-(2)$ monopole} (Japanese),
private note,
http://kansai-gauge.squares.net/misc/pin2-ori.pdf
 
 
 
\bibitem{SSGR}{K.Shiozaki}, {H.Shapourian}, {K.Gomi}, {S.Ryu},
 {\em Many-body topological invariants for fermionic short-range entangled topological phases protected by antiunitary symmetries},
  {Physical Review B},
  {98},
  {3},
  {035151},
  (2018),
  {APS}

 
\bibitem{Stoltz}
S.Stoltz.
{\em Exotic structures on 4-manifolds detected by spectral invariants}.
{Inventiones mathematicae},
  {\bf 94},
  no.{1},
  :{147--162},
  (1988),
  Springer.
  
\bibitem{Witten2}
E. Witten.
{\em Supersymmetry and {M}orse theory.}
J. Differential Geom., {\bf 17}(4):661--692, (1982).

  
 \bibitem{Zh93}
W.Zhang.
{\em Spin{$^c$}-manifolds and Rokhlin congruences.}
C. R. Acad. Sci. Paris S\'er. I Math., {\bf 317}(7):689--692, (1993).

  
\bibitem{Zhang}
W.Zhang,
{\em A mod $2$ index theorem for $Pin^-$ manifolds},
{Science China Mathematics},
  {\bf60},
  no.{9},
  :{1615--1632},
  (2017),
  Springer.

 
\end{thebibliography}
\end{document}